\documentclass[12pt,twoside,a4paper]{amsart}
\usepackage[margin=2.0cm]{geometry}
\usepackage{amsmath,amssymb,amsfonts,mathtools,mathrsfs}
\usepackage{enumerate,enumitem}
\usepackage{microtype}
\usepackage[colorlinks=true,linkcolor=blue,citecolor=red,urlcolor=cyan]{hyperref}
\usepackage[normalem]{ulem}
\usepackage{tikz}

\makeatletter
\let\disSphOriginalBigstar\bigstar
\newcommand{\disSphScaledBigstar}[2]{%
  \vcenter{\hbox{\scalebox{0.75}{$\m@th#1\disSphOriginalBigstar$}}}%
}
\renewcommand{\bigstar}{%
  \mathord{\mathpalette\disSphScaledBigstar\relax}%
}
\@namedef{subjclassname@2020}{\textup{2020} Mathematics Subject Classification}
\makeatother

\numberwithin{equation}{section}

\theoremstyle{plain}
\newtheorem{theorem}{Theorem}[section]
\newtheorem{proposition}[theorem]{Proposition}
\newtheorem*{propA}{Proposition A}
\newtheorem{lemma}[theorem]{Lemma}
\newtheorem{corollary}[theorem]{Corollary}

\theoremstyle{definition}

\newtheorem{remark}[theorem]{Remark}

\newcommand{\R}{\mathbb R}
\newcommand{\Z}{\mathbb Z}
\newcommand{\N}{\mathbb N}
\newcommand{\T}{\mathbb T}

\newcommand{\wh}{\widehat}

\newcommand{\Ent}{\mathscr E}
\newcommand{\supp}{\operatorname{supp}}

\newcommand{\dd}{\,d}
\newcommand{\one}{\mathbf 1}
\newcommand{\eps}{\varepsilon}

\newcommand{\Nmu}{\mathcal{N}_\mu}

\newcommand{\MHL}{M_{\mathrm{HL}}}
\newcommand{\cMSW}{\mathfrak c_d}
\newcommand{\Clog}{C_{\log}}

\title[Endpoint estimates for discrete spherical averages]
{On endpoint estimates for discrete spherical averages over sparse sequences}

\author{Sanghyuk Lee}
\address{Department of Mathematical Sciences and RIM, Seoul National University, Seoul 08826, Republic of Korea}
\email{shklee@snu.ac.kr}

\author{Ji Li}
\address{School of Mathematical and Physical Sciences, Macquarie University, NSW, 2109, Australia.}
\email{ji.li@mq.edu.au}

\author{Chong-wei Liang}
\address{Department of Mathematics, National Taiwan University, Taiwan.}
\email{d10221001@ntu.edu.tw}

\author{Chun-Yen Shen}
\address{Department of Mathematics, National Taiwan University, Taiwan.}
\email{cyshen@math.ntu.edu.tw}


\begin{document}
\begin{abstract}
Let $d\ge5$, and let $A_\lambda$ denote the normalized discrete
spherical average on $\Z^d$ with squared radius $\lambda$.
For a strictly increasing sequence $(\mu_k)$ of positive integers,
set $\lambda_k=\mu_k!$ and consider the lacunary maximal operator
$A_\star f:=\sup_k|A_{\lambda_k}f|$.
For each fixed $\gamma>0$, we prove that the estimate
\[
 \#\{x:A_\star f(x)>\alpha\}
 \le K\sum_x\frac{|f(x)|}{\alpha}
       \left(1+\log^+\frac{|f(x)|}{\alpha}\right)^\gamma
\]
holds with some finite $K$ for all $\alpha>0$ and finitely supported $f$
if and only if
\[
 \sup_{N\ge2}\frac{\#\{k:\mu_k\le N\}}{(1+\log N)^\gamma}<\infty.
\]
The summand is interpreted as zero when $f(x)=0$.
The case $\gamma=1$ answers the $\ell\log\ell$ endpoint question of
Kesler, Lacey and Mena in the formulation stated above.
For $\mu_k=2^{2^k}$, these estimates hold for every fixed $\gamma>0$,
but $A_\star$ is not of weak type $(1,1)$.

For factorial blocks formed from arbitrary subsets of dyadic grids
at widely separated scales, we characterize restricted weak type
$(p,p)$, $1<p\le d/(d-1)$, by a uniform covering condition.
In particular, for each $1<p_*\le d/(d-1)$, we construct a set of squared radii whose maximal operator is of restricted weak
type $(p_*,p_*)$, of strong type $(p,p)$ for $p>p_*$, and unbounded
on $\ell^p$ for $1\le p<p_*$.
\end{abstract}
\maketitle

\section{Introduction and main results}
Let $d\ge5$. For an integer $\lambda\in\N$, we define the normalized discrete spherical average
\[
A_{\lambda}f(x):=\frac{1}{s_\lambda}\sum_{\substack{n\in\Z^d\\|n|^2=\lambda}} f(x-n),
\]
 where $s_\lambda:=\#\{n\in\Z^d:\ |n|^2=\lambda\}$ denotes the number of lattice points on the sphere of radius $\lambda^{1/2}$. For $d\geq5$, one has
\[
 s_\lambda\simeq_d\lambda^{d/2-1},
 \quad
\lambda\in\N,
\]
see \cite{MR0097357} and \cite[p.~190 and (3.1)]{MR1888798}.

We study two endpoint problems for different sets of squared radii.
At $p=1$, we consider factorial sequences $\lambda_k=\mu_k!$ and
characterize the endpoint estimates by counting the parameters
$\mu_k$. For $1<p\le p_0$, where $p_0= d/(d-1)$,
we consider factorial blocks and characterize restricted weak type
estimates by covering numbers within each block.

For the major-arc approximation, we also use the normalization
\[
 \overline A_\lambda f(x)
 :=\lambda^{1-d/2}\sum_{\substack{n\in\Z^d\\|n|^2=\lambda}}f(x-n)\quad\text{and}\quad 
 \rho_\lambda:=\frac{\lambda^{d/2-1}}{s_\lambda}.
\]
Then, with the normalization, 
\[
 A_\lambda=\rho_\lambda\overline A_\lambda\quad\text{and}\quad 
 R_d^{\rm norm}
 :=\sup_{\lambda\in\N}
 \max\{\rho_\lambda,\rho_\lambda^{-1}\}<\infty.
\]
Only this uniform comparability is used.

\subsection{The endpoint at \texorpdfstring{$p=1$}{p=1}}
We first consider the maximal operator associated with a lacunary sequence of radii $\{\lambda^{1/2}_k\}$:
\[
A_\star f(x):=\sup_k |A_{\lambda_k}f(x)|.
\]
For the factorial sequence $\lambda_k=(2^k)!$, the maximal operator
$A_\star$ is bounded on $\ell^p(\Z^d)$ for every $1<p<\infty$
\cite{MR4470175}. More generally, for factorial sequences
$\lambda_k=\mu_k!$ with $\mu_k\uparrow\infty$, Kesler, Lacey and Mena
\cite{MR3982252} proved that if
\begin{equation}\label{eq:KLM}
  \lim_{k\to\infty}\frac{\log\mu_k}{\log k}=\infty,
\end{equation}
then $A_\star$ is bounded on $\ell^p(\Z^d)$ for every $1<p<\infty$. Repeated values of $\mu_k$ do not affect the maximal operator, so throughout
we delete repetitions and assume without loss of generality that $(\mu_k)$ is a
\emph{strictly increasing} sequence of positive integers.
For general lacunary sequences, the bounds in
\cite[Theorem 1.1]{MR3982252} cover a smaller range of $p$;
see also \cite{MR3996030}. Arithmetic obstructions for lacunary sequences were exhibited by
Zienkiewicz and reported by Hughes \cite{MR3996030}.
Cook and Hughes \cite{CookHughes21} extended this construction
to a class of homogeneous forms. Hughes \cite{Hughes26} also discusses $\ell^p$ boundedness
for general sets of squared radii in relation to their
arithmetic structure and counting growth.

Kesler, Lacey and Mena \cite[Remark~(3), Sec.~2]{MR3982252}
asked whether the maximal operators in their Theorem~1.2 map
$\ell\log\ell$ into weak $\ell^1$. We study this question in the formulation specified below.

We give \emph{a necessary and sufficient condition} for each fixed
logarithmic estimate along factorial sequences. For
$\gamma=1$, this characterizes when the
$\ell\log\ell$ estimate holds. The condition
\eqref{eq:KLM} alone is insufficient for that endpoint estimate;
the condition is logarithmic counting of the factorial parameters.
For a strictly increasing sequence $(\mu_k)$, define the counting function
\begin{align*}
  \mathcal{N}_\mu(N)=\#\{k\geq1:\mu_k\leq N\},\qquad N\geq2
\end{align*}
and the logarithmic counting constant
\begin{equation}\label{eq:Cmu}
  C_{\log}=\sup_{N\geq2}\frac{\Nmu(N)}{1+\log N}.
\end{equation}
All logarithms are natural unless another base is indicated. For all $t>0$, write $\mathcal{L}(t)=1+\log^+t$ and $\mathcal{L}(0)=1.$
Here, $\log^+t=\max\{0,\log t\}$.

For $\gamma>0$, put $\Phi_\gamma(t)=t\mathcal L(t)^\gamma$ for $t\ge0$
and define
\begin{equation}\label{eq:Clog-gamma}
 C_{\log,\gamma}
 :=\sup_{N\ge2}\frac{\Nmu(N)}{(1+\log N)^\gamma}\quad\text{and}\quad C_{\log,1}=C_{\log}.
\end{equation}
The endpoint estimate is formulated with the formula
$\sum_x\Phi_\gamma(|f(x)|/\alpha)$ at each level $\alpha>0$.
Thus $\alpha$ occurs both in the level set and inside $\Phi_\gamma$;
this normalization is part of the assertion and is not shorthand for
a bound by $\alpha^{-1}$ times a Luxemburg norm.

\subsubsection{Sufficient condition for the endpoint estimate}
Our first main result gives a sufficient condition for each endpoint estimate, including the $\ell\log\ell$ estimate when $\gamma=1$.

\begin{theorem}\label{thm:full-mu}
Let $d\ge5$ and let $\lambda_k=\mu_k!$, where $(\mu_k)$ is a strictly
increasing sequence of positive integers. Fix $\gamma>0$ and suppose
$C_{\log,\gamma}<\infty$. Then, for every $\alpha>0$ and every finitely
supported $f:\Z^d\to\mathbb C$, writing $t_x=|f(x)|/\alpha$, one has
\begin{align}\label{eq:gamma-modular}
 \#\{x:A_\star f(x)>\alpha\}
 &\le C_d\left[1+\gamma^{-1}
       +(1+24\log2)^\gamma C_{\log,\gamma}\right]
       \sum_x t_x\mathcal L(t_x)^\gamma.
\end{align}
The constant $C_d$ is independent of $\gamma$, the sequence, $f$, and
$\alpha$. In particular, at $\gamma=1$ this gives the bound
$C_d(1+C_{\log})\sum_x t_x\mathcal L(t_x)$.
\end{theorem}

\begin{remark}
The logarithmic counting condition $C_{\log}<\infty$ is stronger
than \eqref{eq:KLM}. Indeed, since $(\mu_k)$ is strictly increasing,
we have $\mathcal N_\mu(\mu_k)=k$. Hence, $k\leq C_{\log}(1+\log\mu_k)$
and therefore, as $k\to\infty$,
\[
  \frac{\log\mu_k}{\log k}
  \geq
  \frac{k/C_{\log}-1}{\log k}
  \longrightarrow\infty.
\]
The converse is false. For example, let $\mu_k=k+\left\lceil \exp({k^{1/2}})\right\rceil$. Then $(\mu_k)$ is strictly increasing and $\log\mu_k\simeq{k}^{1/2}$, so \eqref{eq:KLM} holds. On the other hand, strict monotonicity gives $\Nmu(\mu_k)=k$.
Therefore, as $k\to\infty$,
\[
 \frac{\Nmu(\mu_k)}{1+\log\mu_k}
 =
 \frac{k}{1+\log\mu_k}
 \longrightarrow\infty,
\]
and consequently $C_{\log}=\infty$.
\end{remark}

\begin{corollary}\label{cor:model}
Suppose $\lambda_k=(2^k)!$, $k\geq1$, and $d\geq5$. Then, for every $\alpha>0$ and every finitely supported $f:\Z^d\to\mathbb C$,
\begin{equation}\label{eq:model-endpoint}
 \#\{x:A_\star f(x)>\alpha\}
 \leq C_d\sum_x\frac{|f(x)|}{\alpha}
 \mathcal{L}\left(\frac{|f(x)|}{\alpha}\right).
\end{equation}
The finite constant $C_d$ is independent of $f$ and $\alpha$.
\end{corollary}

We obtain Theorem~\ref{thm:full-mu} from an estimate which does not require
\eqref{eq:KLM} or logarithmic counting.

\begin{theorem}\label{thm:counting}
Let $d\geq5$ and let $\lambda_k=\mu_k!$, where $(\mu_k)$ is a strictly increasing sequence of positive integers. For every $\alpha>0$ and every finitely supported $f:\Z^d\to\mathbb C$,
\begin{equation}\label{eq:counting-endpoint}
 \#\{x:A_\star f(x)>\alpha\}
 \leq C_d\sum_x\frac{|f(x)|}{\alpha}
 \left[
 \log\!\left(e+\mathcal{L}\!\left(\frac{|f(x)|}{\alpha}\right)\right)
 +\Nmu\!\left(2^{24\mathcal{L}(|f(x)|/\alpha)}\right)
 \right].
\end{equation}
The constant $C_d$ is independent of the sequence, $f$, and $\alpha$.
\end{theorem}

The exponent $24$ is a fixed choice in the proof. We make no claim
that this numerical constant is best possible. 

\subsubsection{Necessary condition for the endpoint estimate}
Our second main result gives the corresponding necessary condition for
each fixed logarithmic exponent.

\begin{theorem}\label{thm:klm-counterexample}
Let $d\ge5$ and let $\lambda_k=\mu_k!$, where $(\mu_k)$ is a strictly
increasing sequence of positive integers. Fix $\gamma>0$. Suppose that,
for some $K<\infty$,
\begin{align}\label{eq:necessary-endpoint}
 \#\{x\in\Z^d:A_\star f(x)>\alpha\}
 \le K\sum_x\frac{|f(x)|}{\alpha}
       \mathcal L\left(\frac{|f(x)|}{\alpha}\right)^\gamma
\end{align}
for every $\alpha>0$ and every finitely supported $f:\Z^d\to\mathbb C$.
Then
\[
 C_{\log,\gamma}
 \le N_*(d)+\frac{C_d[4(d-1)]^\gamma}{1-2^{-\gamma}}K<\infty,
\]
where $N_*(d)$ and $C_d$ depend only on $d$. In particular, at
$\gamma=1$ this gives $C_{\log}\le C_d(1+K)$.
\end{theorem}

Together, Theorems~\ref{thm:full-mu} and~\ref{thm:klm-counterexample}
show that, for each fixed $\gamma>0$, the estimate
\eqref{eq:necessary-endpoint} holds with some finite $K$ if and only if
$C_{\log,\gamma}<\infty$. The constants need not remain bounded as
$\gamma\downarrow0$. These estimates do not imply weak type $(1,1)$:
Corollary~\ref{nec:all-positive-not-weak} shows that the sequence
$\mu_k=2^{2^k}$ satisfies every positive logarithmic estimate, while
its maximal operator fails weak type $(1,1)$.

\subsubsection{Outline of the necessity proof}
 The necessity proof relies on the asymptotic behavior of the distribution of lattice points on spheres modulo a prime $p$. To motivate the proof of Theorem \ref{thm:klm-counterexample} in the case $\gamma=1$, we present a preliminary proposition with a weaker counting conclusion. We prove this proposition in Section~\ref{sec:counterexample}.
\begin{propA}
Let $d\geq5$ and let $\lambda_k=\mu_k!$, where $(\mu_k)$ is a strictly increasing sequence of positive integers. Suppose that, for some $C_0<\infty$,
\begin{align}
 \#\{x\in\Z^d:A_\star f(x)>\alpha\}
 \leq C_0\sum_x\frac{|f(x)|}{\alpha}
\mathcal{L}\left(\frac{|f(x)|}{\alpha}\right)
\end{align}
for every $\alpha>0$ and every finitely supported $f:\Z^d\to\mathbb C$. Then
\begin{align}\label{weakupper}
 \Nmu(N)\leq C_{d,C_0}(1+\log N)^2,
 \qquad N\geq2.
\end{align}
\end{propA}

In the proof of Proposition A, integrating the distribution
bound on the finite quotient introduces the extra logarithm in
\eqref{weakupper}. The sharp argument instead tests a common level
at which each selected radius contributes $\gtrsim_d p^{d-1}$ residue
classes. This follows from an \emph{asymptotic formula} for the
distribution of lattice points within a fixed residue class modulo $p$.
Keeping the number of discarded small parameters in the resulting
counting recursion also allows us to treat every $0<\gamma<1$.

 For a positive integer $\lambda$ and a prime $p$, we define
\[
 \nu_{\lambda,p}(a)
 =\frac1{s_\lambda}
 \#\{n\in\Z^d:|n|^2=\lambda,\ n\equiv a\pmod p\},
 \qquad a\in\Z_p^d.
\]
\begin{theorem}\label{thm:asym}
Let $d\ge5$, let $\lambda$ be a positive integer, and let $p$ be an
odd prime with $p\nmid\lambda$. If $\lambda\ge p^8$,
then, uniformly for $a\in\Z_p^d$,
\begin{equation}\label{nec:modular-asymptotic}
 \nu_{\lambda,p}(a)
 =\rho_\lambda\cMSW p^{1-d}\one_{\{|a|^2\equiv\lambda\pmod p\}}
 \sum_{\substack{r\geq 1\\(p,r)=1}}\mathfrak{QR}_r(\lambda)+O_d(\lambda^{-1/8}),
\end{equation}
where $\mathfrak{QR}_r(\lambda)
 =\sum_{a\in\Z_r^\times}e_r(-a\lambda)G(a,0,r)$ and $G(a,\ell,r)= r^{-d} \sum_{n\in \Z_r^d} e_r(a|n|^2 + n \cdot \ell)$ is the normalized Gauss sum.
\end{theorem}

Related lower bounds for the number of points in a fixed residue
class were obtained by Cook and Hughes \cite{CookHughes21} for
a class of homogeneous forms. Their result selects a suitable
level for each residue class, whereas Theorem~\ref{thm:asym}
gives an asymptotic formula for every squared radius satisfying
the stated conditions, with an error bound uniform in the
prime modulus.

\subsection{Restricted weak type estimates for \texorpdfstring{$1<p\le p_0$}{1<p<=p0}}
We now turn to a different family of squared radii, formed by factorial
blocks, and study restricted weak type estimates for $1<p\le p_0$.

For continuous spherical averages, Seeger, Wainger and Wright
\cite{MR1329463} determined the nonendpoint range in terms of the covering
numbers of the dilation set. Seeger, Tao and Wright \cite{MR1955209} obtained
endpoint results and, in particular, a restricted weak type criterion
which requires no additional regularity assumption. We establish the
corresponding restricted weak type characterization for the factorial
blocks of squared radii defined below.

 For $\Lambda\subset\N$, let
\[
 M_\Lambda f(x)=\sup_{\lambda\in\Lambda}|A_\lambda f(x)|.
\]
Let $\mathbf F=(F_j)_{j\ge1}$ be a family of subsets of $[1,2)$ satisfying
\[
 F_j\subset[1,2)\cap2^{-j}\Z,\quad\forall j\in\N.
\]
Empty blocks are allowed. Consider
\begin{equation}\label{eq:block-family}
 n_j=2^{j^2},\quad R_j=2^j n_j!\quad\text{and}\quad
 \Lambda(\mathbf F)=\bigcup_{j\ge1}\{(R_jt)^2:t\in F_j\}.
\end{equation}
All the squared radii in \eqref{eq:block-family} are integers.
The exponent $d/(d-2)$ in the theorem of Magyar, Stein and Wainger
\cite{MR1888798} concerns strong bounds for the maximal operator over
all integer squared radii. Indeed, for the point mass $\delta_0$,
$M_{\N}\delta_0(x)=s_{|x|^2}^{-1}\simeq_d|x|^{2-d}$ when $x\ne0$,
which belongs to $\ell^p$ only for $p>d/(d-2)$; see
\cite[Sec.~8]{MR1888798}. At the endpoint $p=d/(d-2)$, Ionescu
\cite[Theorem~1]{Ionescu04} proved that $M_{\N}$ is of restricted
weak type $(p,p)$. In \eqref{eq:block-family}, even the full
grid has consecutive actual radii separated by $n_j!$, so its
point-mass test samples a different set of lattice points.
The value $p_0=d/(d-1)$ is the upper endpoint in the continuous
restricted weak type criterion used in our proof. For the blocks
\eqref{eq:block-family}, its covering condition is expressed in
terms of the sets $F_j$.

For a bounded set $F\subset\R$, let $N(F,\delta)$ be the least number
of intervals of length $\delta$ needed to cover $F$, with
$N(\varnothing,\delta)=0$. For $\beta>0$, define
\begin{align*}
 \Ent_\beta(\mathbf F)
 =\sup_{j\ge1}\sup_{0<\delta<1}\delta^\beta N(F_j,\delta).
\end{align*}

\begin{theorem}\label{thm:restricted}
Let $d\ge5$, $1<p\le p_0$, and $\beta=(d-1)(p-1)$.
For the set \eqref{eq:block-family}, the restricted weak type estimate
\begin{equation}\label{eq:restricted-type}
 \#\{x:M_{\Lambda(\mathbf F)}\one_H(x)>\alpha\}
 \le C\alpha^{-p}\#H
\end{equation}
holds for every finite $H\subset\Z^d$ and every $\alpha>0$ if and
only if $\Ent_\beta(\mathbf F)<\infty$.
Equivalently, $M_{\Lambda(\mathbf F)}$ maps $\ell^{p,1}$ to
$\ell^{p,\infty}$.
\end{theorem}

The number $p_0$ is the fixed upper endpoint of our range, whereas
$p_*$ below denotes the critical exponent of a chosen family.
Interpolation and Theorem~\ref{thm:restricted} give the following
corollary.
\begin{corollary}\label{cor:critical}
Let $1<p_*\le p_0$ and $\beta=(d-1)(p_*-1)$.
The operator $M_{\Lambda(\mathbf F)}$ is of restricted weak type
$(p_*,p_*)$ and is unbounded on every $\ell^p$, $1\le p<p_*$, if
and only if
\begin{equation}\label{eq:critical-conditions}
 \Ent_\beta(\mathbf F)<\infty
 \quad\text{and}\quad \Ent_\theta(\mathbf F)=\infty\quad(0<\theta<\beta).
\end{equation}
Under these conditions it is bounded on every $\ell^p$, $p>p_*$,
and fails restricted weak type $(p,p)$ for $1<p<p_*$.
\end{corollary}

The binary construction \eqref{rw:binary-block} realizes the conditions
of Corollary~\ref{cor:critical} for every $1<p_*\le p_0$.
The conclusion at the critical exponent is restricted weak type.
The blocks in \eqref{eq:block-family} are widely separated, although
the radii within a block need not be lacunary.

The factorial structure enters the proofs through divisibility and
growth. The algebraic regrouping of the $Q$-Gauss sums in
\eqref{eq:B-grouped} holds independently of $\lambda$. On the factorial
tail, divisibility by the small denominators makes the phases
$e_q(-\lambda a)$ equal to one. The common arithmetic factor $U_Q$
has a kernel with uniformly bounded $\ell^1$ norm
(Lemma~\ref{lem:U-bound}).
For factorial blocks, rapid growth controls the approximation error
and the number of initial radii. Both estimates account for all
radii within each block (Lemma~\ref{lem:block-tail}).
The endpoint condition that remains is parameter counting for the
factorial sequence at $p=1$, and covering geometry within the blocks
for $1<p\le p_0$.

\subsection{Organization}
The paper is organized as follows. In Section~\ref{sec:prelim}, we
review the circle method and the strong and weak sampling estimates,
including weak sampling at $p=1$. In Section~\ref{sec:3}, we prove the
quantitative superlacunary estimate and the effective MSW decomposition,
and then establish a separate endpoint decomposition. In
Section~\ref{sec:endpoint}, weak type summation gives
Theorems~\ref{thm:counting} and~\ref{thm:full-mu}, and
Corollary~\ref{cor:model}. In Section~\ref{sec:counterexample}, periodic
tests and residue-class asymptotics yield the counting recursion for
Theorem~\ref{thm:klm-counterexample}. We also exhibit a sequence
whose maximal operator satisfies the endpoint estimate
for every fixed $\gamma>0$ but is not of weak type $(1,1)$. In Section~\ref{sec:6},
we prove Theorem~\ref{thm:restricted} and Corollary~\ref{cor:critical},
and construct the binary examples with prescribed critical exponents.

\medskip

\section{Preliminaries}\label{sec:prelim}
\subsection{Notation}
Throughout the paper, $e(t)$ denotes the phase $e^{2\pi i t}$. For $q\ge1$, we write $e_q(t)=e(t/q)$, $\Z_q=\Z/q\Z$ and $\Z^\times_q=(\Z/q\Z)^\times$ for the group of units modulo $q$. The size of this group is given by Euler's totient function $\varphi(q)$.
For suitable functions $f:\Z^d\to\mathbb C$ and
$F:\R^d\to\mathbb C$, we write $\wh f$ and $\widetilde F$
for their Fourier transforms on $\Z^d$ and $\R^d$, respectively.
We use $*$ and $\bigstar$ for convolution on $\Z^d$ and $\R^d$,
respectively.

Unless a dependence is displayed explicitly, implicit constants may depend only on $d$ and on the fixed cutoff functions.  In particular, they are independent of $p,L,N,Q$, every finite index set $J$ and its cardinality, and the sequence; dependence through quantities such as $C_{\log}$ is always displayed.

For Lebesgue or counting measure and $1\le p<\infty$, we use the
Lorentz functionals
\[
 \|h\|_{p,1}=\int_0^\infty|\{|h|>s\}|^{1/p}\,ds\quad\text{and}\quad 
 \|h\|_{p,\infty}=\sup_{s>0}s|\{|h|>s\}|^{1/p}.
\]
The first is the $L^1$ or $\ell^1$ norm when $p=1$; for vector-valued
functions, $|h|$ denotes the pointwise norm. Fixed-$p$ constants in
the weak sampling estimates are written as $C_{d,p}$ and are
independent of the number of components.

For a countable family $(T_k)_{k\in I}$ and a finite set $J\subset I$, write
\[
 T_Jf:=\sup_{k\in J}|T_kf|.
\]
Maximal estimates for such families are first proved with constants independent of $J$ and $\#J$.  If $J_m\uparrow I$, then $T_{J_m}f\uparrow\sup_{k\in I}|T_kf|$; strong norms pass to the full supremum by monotone convergence, and distribution estimates by continuity from below of counting measure.  For the probability spherical averages we also use
\[
 A_{\star,J}f:=\sup_{k\in J}|A_{\lambda_k}f|
 \leq A_{\star,J}(|f|),
\]
which follows from the positivity of each $A_{\lambda_k}$.

\subsection{The circle method and sampling argument}
\subsubsection{The circle method decomposition}
For each $\lambda\in\mathbb N$, let $M_\lambda$ be the convolution
operator on $\Z^d$ defined by
\begin{align}\label{majoprarc}
M_\lambda:=\cMSW\sum_{q=1}^{\infty} \sum_{a \in \Z_q^\times} e_q(-\lambda a) M^{a/q}_\lambda,
\end{align}
where the sum is taken over all reduced fractions $a/q$ with
$0<a/q\leq1$, and $\cMSW=\pi^{d/2}/\Gamma(d/2)$.
Each $M^{a/q}_\lambda$ is a convolution operator with Fourier
multiplier $m^{a/q}_\lambda$ given by
\begin{align}\label{majorarc2}
\sum_{{\ell \in \Z^d}} G(a,\ell,q) \, \widetilde{\psi}_q\left(\xi - \frac{\ell}{q}\right) \widetilde{d\sigma_\lambda}\left(\xi - \frac{\ell}{q}\right),
\end{align}
where $G(a,\ell,q)$ is the normalized Gauss sum
\begin{align}\label{normalizedGauss}
q^{-d} \sum_{n\in \Z_q^d} e_q(a|n|^2 + n \cdot \ell).
\end{align}
We set $\widetilde{\psi}_q(\cdot) = \widetilde{\psi}(q\cdot)$, where
$\widetilde{\psi} \in C_c^\infty(\mathbb R^d)$ is a smooth cut-off
function supported near the origin and satisfying
$\mathbf{1}_{\{|\xi| \le 1/4\}} \le \widetilde{\psi}(\xi) \le \mathbf{1}_{\{|\xi| \le 1/2\}}$.
The function $\widetilde{d\sigma_\lambda}$ is the Fourier transform
of the normalized surface measure on the sphere of radius
$\lambda^{1/2}$. If $\gcd(a,q)=1$, then
$|G(a,\ell,q)|\lesssim_dq^{-d/2}$.

According to
\cite[equations~(3.1)--(3.4), pp.~198--199]{MR1888798}, for each
$\lambda\in\N$ we have
\begin{equation}\label{eq:MSW_decomp}
 \overline A_\lambda=M_\lambda+E_\lambda.
\end{equation}
Here $E_\lambda=\overline A_\lambda-M_\lambda$ is the approximation error.
The proof of Proposition~4.1 in \cite[Sections~5--6]{MR1888798}
also applies to the fixed cutoff above: its translates have bounded
overlap, and $1-\widetilde\psi_q$ vanishes for $|\xi|\le1/(4q)$.
That proposition uses the radius variable $R$ and decay $R^{2-d/2}$.
The range $\Lambda\le\lambda\le2\Lambda$ is contained in
$\sqrt\Lambda\le R\le2\sqrt\Lambda$, so it gives the following
$\ell^2$ estimate for our normalization $\overline A_\lambda$:
\begin{proposition}\label{minor}
    For any $\Lambda>0$ and $d\geq5$,
    \begin{align*}
        \left\|\sup_{\Lambda\leq\lambda\leq2\Lambda}\left|E_\lambda f\right|\right\|_{\ell^2}\lesssim_d \Lambda^{1-d/4}\left\|f\right\|_{\ell^2}.
    \end{align*}
\end{proposition}

We will also need the $\ell^2$ estimate for the maximal function  $M^{a/q}_\star f:=\sup_{\lambda>0}|M^{a/q}_\lambda f|$:
\begin{proposition}\label{othermaxi}
For all $d\geq3$,
\begin{equation}\label{eq:single-denom-l2}
  \left\|M^{a/q}_\star f\right\|_{\ell^2}
  \lesssim_d q^{-d/2}\|f\|_{\ell^2}.
\end{equation}
\end{proposition}

\subsubsection{The sampling argument}
We recall the strong sampling theorem from \cite[Section~2]{MR1888798}.
Let $T$ be a convolution operator on $\R^d$ with kernel $K$ and bounded
Fourier multiplier $m$ supported in $(-1/2,1/2]^d$.
The restriction $K_{dis}:=K|_{\Z^d}$ is well-defined and gives the
discrete convolution operator
\begin{align*}
    T_{dis}(f)(n)=f*K_{dis}(n):=\sum_{m\in\Z^d}K(m)f(n-m).
\end{align*}
Its Fourier multiplier is the periodization
$m_{per}(\xi):=\sum_{\ell\in\Z^d}m(\xi-\ell)$, so that
\begin{align*}
    \widehat{T_{dis}(f)}(\xi)=m_{per}(\xi)\cdot\wh f(\xi)
\end{align*}
for a suitable function $f$ on $\Z^d$. In \cite[Proposition 2.1]{MR1888798}, the following sampling estimate is proved:

\begin{proposition}\label{SAM1}
Let $1\leq p\leq\infty$ and let $B_1,B_2$ be finite-dimensional Banach
spaces.  Suppose that the bounded measurable multiplier $m$ takes values in
$\mathcal L(B_1,B_2)$, is supported in $(-1/2,1/2]^d$, and its continuous
operator $T$ is bounded from $L^p_{B_1}(\R^d)$ to $L^p_{B_2}(\R^d)$.  Then
$T_{dis}$ is bounded from $\ell^p_{B_1}(\Z^d)$ to
$\ell^p_{B_2}(\Z^d)$, and
\begin{align*}
 \left\|T_{dis}\right\|_{\ell^p_{B_1}(\Z^d)\to\ell^p_{B_2}(\Z^d)}
 \leq C_d
 \left\|T\right\|_{L^p_{B_1}(\R^d)\to L^p_{B_2}(\R^d)}.
\end{align*}
The constant $C_d$ is independent of $p,B_1$, and $B_2$.
\end{proposition}

 The proof of this proposition relies on the specific choice of the sampling function on $\mathbb R^d$
\begin{align*}
\Psi(x):=\left(\frac{\sin\pi x_1}{\pi x_1}\right)^2\ldots\left(\frac{\sin\pi x_d}{\pi x_d}\right)^2,\quad\text{where}\,\,x=(x_1,\ldots,x_d)\in\R^d
\end{align*}
and the following extension principle. Let $B$ be a finite-dimensional
Banach space. For any suitable function $f:\Z^d\rightarrow B$, define
its extension $F=f_{ext}$ to $\mathbb R^d$ by
\begin{align*}
 F(x)=f_{ext}(x):=\sum_{n\in\Z^d}f(n)\Psi(x-n).
\end{align*}
Then we have the norm equivalence
\begin{align}\label{contidiscrete}
\left\|f\right\|_{\ell^p_{B}(\Z^d)}\simeq_{d} \left\|F\right\|_{L^p_{B}(\R^d)}.\end{align}

We also use a rescaled form. Fix an integer $q\geq1$ and suppose that $m$ is supported in $(-1/(2q),1/(2q)]^d$. Define
\[
  m^q_{per}(\xi):=\sum_{\ell\in\Z^d}m(\xi-\ell/q).
\]
Let $T^q_{dis}$ be the corresponding convolution operator with Fourier multiplier $m^q_{per}$. The following property of $T^q_{dis}$ is established in \cite[Corollary~2.1]{MR1888798}.

\begin{proposition}\label{sampling}
Let $q\geq1$ and suppose, in addition to the hypotheses of Proposition~\ref{SAM1}, that $m$ is supported in
$(-1/(2q),1/(2q)]^d$.  Then
\begin{align*}
 \left\|T^q_{dis}\right\|_{\ell^p_{B_1}(\Z^d)\to\ell^p_{B_2}(\Z^d)}
 \leq C_d
 \left\|T\right\|_{L^p_{B_1}(\R^d)\to L^p_{B_2}(\R^d)}.
\end{align*}
The constant $C_d$ is independent of $q,p,B_1$, and $B_2$.
\end{proposition}

\medskip

\subsubsection{Sampling in Lorentz spaces}
We prove the endpoint sampling statement needed below, including $p=1$.
Here $\MHL$ denotes the Hardy--Littlewood maximal operator. 

\begin{lemma}\label{lem:bandlimited}
Let $1\le p<\infty$ and $\{g_j\}_{j\in J}$ be a finite family of
Schwartz functions with Fourier supports in $[-1/4,1/4]^d$. Set $G=\sup_{j\in J}|g_j|$. Then
$$ \|\{G(n)\}_{n\in\Z^d}\|_{\ell^{p,\infty}}
 \le C_{d,p}\|G\|_{L^{p,\infty}}.$$
\end{lemma}
\begin{proof}
Let $r=1/2$ and $D=2d+2$. For any Schwartz function with Fourier support in $[-1/4,1/4]^d$,
define
\[
P_Dg(x)=\sup_{y\in\R^d}\frac{|g(x-y)|}{(1+|y|)^D}.
\]
We claim the estimate 
\begin{equation}\label{eq:Peetre}
 P_Dg(x)\le C_d[\MHL(|g|^r)(x)]^{1/r}.
\end{equation}
Choose an even Schwartz function $\eta$ with $\widetilde\eta=1$ on
$[-1/4,1/4]^d$, so that $g=\eta\bigstar g$. Put
$P=P_Dg(x)$, which is finite since $P\le\|g\|_\infty$.
If $P>0$, then the reproduction formula and
$|g(x-y)|^{1-r}\le P^{1-r}(1+|y|)^{D(1-r)}$ give
\begin{align*}
 \frac{|g(x-z)|}{(1+|z|)^D}
 &\le P^{1-r}\int |\eta(z-y)|(1+|z-y|)^D(1+|y|)^{-Dr}|g(x-y)|^r\,dy\\
 &\le C_dP^{1-r}\int(1+|y|)^{-Dr}|g(x-y)|^r\,dy.
\end{align*}
Here, we used $(1+|y|)/(1+|z|)\le1+|z-y|$.
Taking the supremum over $z$, dividing by $P^{1-r}$, and summing
over dyadic annuli proves \eqref{eq:Peetre}, since $Dr>d$.
The case $P=0$ is immediate.

Applying \eqref{eq:Peetre} to each $g_j$ gives
\[
 \sup_jP_Dg_j(x)\le C_d[\MHL(G^r)(x)]^{1/r}.
\]

Let $Q_n=n+[-1/2,1/2)^d$. Then, for each $x\in Q_n$, we have
\[
 \sup_{j\in J}P_Dg_j(x)
 \ge\frac{G(n)}{(1+|x-n|)^D}
 \ge(1+\sqrt d/2)^{-D}G(n).
\]
It follows that there is $c_d>0$ such that, for all $s>0$, 
\[
 \bigsqcup_{\{n:G(n)>s\}}Q_n
 \subset\{x:\MHL(G^r)(x)>c_ds^r\}.
\]
Take $q=p/r=2p>1$. A change of variable gives $
 \|G^r\|_{L^{q,\infty}}
 =\|G\|_{L^{p,\infty}}^r$.
The cubes $Q_n$ are disjoint and have volume one. The maximal operator
$\MHL$ is bounded on $L^{q,\infty}$ for $q>1$, so this also applies at
$p=1$, where $q=2$. Applying this bound, we conclude that
\begin{align*}
 \#\{n:G(n)>s\}
 &\leq\left|\{\MHL(G^r)>c_ds^r\}\right|\\
 &\leq C_{d,p}s^{-rq}\|G^r\|_{L^{q,\infty}}^q
 =C_{d,p}s^{-p}\|G\|_{L^{p,\infty}}^p.
\end{align*}
Taking the supremum over $s>0$ gives the desired result.
The constants are independent of $J$. 
\end{proof}

\begin{theorem}\label{thm:sampling}
Let $1\le p<\infty$, let $Q\ge1$ be an integer, and let
$(m_j)_{j\in J}$ be a finite family of smooth multipliers supported
in $[-1/(4Q),1/(4Q)]^d$. Suppose that
\[
 \left\|\sup_{j\in J}|\mathcal F^{-1}(m_j\widetilde g)|\right\|_{L^{p,\infty}}
 \le W\|g\|_{L^{p,1}}.
\]
Let $T_j^Q$ be the discrete operator with multiplier
$(m_j)^Q_{per}(\xi)=\sum_{\ell\in\Z^d}m_j(\xi-\ell/Q)$. Then
\begin{equation}\label{eq:weak-sampling}
 \left\|\sup_{j\in J}|T_j^Qh|\right\|_{\ell^{p,\infty}}
 \leq C_{d,p}W\|h\|_{\ell^{p,1}},
\end{equation}
uniformly in $Q$ and $J$. 
\end{theorem}
\begin{proof}
First, consider $Q=1$ and a finitely supported function $h$ on $\Z^d$.
Choose a fixed Schwartz function $\eta$ with
$\widetilde\eta=1$ on $[-1/4,1/4]^d$, and set
\[
 Eh(x)=\sum_{n\in\Z^d}h(n)\eta(x-n).
\]
The extension operator $E$ maps $\ell^1$ to $L^1$ with norm at most
$A_1=\|\eta\|_1$, and $\ell^\infty$ to $L^\infty$ with norm at most
$A_\infty=\sup_x\sum_n|\eta(x-n)|<\infty$.
For $1<p<\infty$, real interpolation with $\theta=1-1/p$ gives
\[
 \|Eh\|_{L^{p,1}}
 \le A_1^{1/p}A_\infty^{1-1/p}\|h\|_{\ell^{p,1}}
 \le C_d\|h\|_{\ell^{p,1}}.
\]
With the Lorentz functionals fixed above, the common normalization
factor between the real interpolation norm and $\|\cdot\|_{p,1}$
cancels on the two sides. At $p=1$ the same bound follows directly
from the $\ell^1\to L^1$ estimate.

Let $K_j=\mathcal F^{-1}[m_j]$. Since
$m_j\widetilde\eta=m_j$, the corresponding continuous functions are
\[
 g_j(x)=K_j\bigstar(Eh)(x)=\sum_{n\in\Z^d}h(n)K_j(x-n).
\]
These are Schwartz functions with the required Fourier support, and
$g_j(n)=T_j^1h(n)$. Lemma~\ref{lem:bandlimited} and the assumed
continuous estimate therefore give
\[
 \left\|\sup_{j\in J}|T_j^1h|\right\|_{\ell^{p,\infty}}
 \le C_{d,p}\left\|\sup_{j\in J}|g_j|\right\|_{L^{p,\infty}}
 \le C_{d,p}W\|h\|_{\ell^{p,1}}.
\]

The general case follows by dilation. Let $Q>1$ be an integer.
Denote the discrete convolution kernel of $T_j^Q$ by $K_j^Q$.
Write each translation as $\ell=a+Qv$, where $a\in\Z_{Q}^d$
and $v\in\Z^d$. The Fourier inversion formula gives
\begin{align}\label{dilationKQ}
 K_j^Q(n)
 =K_j(n)\sum_{a\in\Z_{Q}^d}
        e^{2\pi i n\cdot a/Q}=Q^dK_j(n)\one_{\{n\in Q\Z^d\}}.
\end{align}
Define
\[
 \dot K_j(x)=Q^dK_j(Qx)\quad\text{and}\quad
 \dot m_j(\xi)=m_j(\xi/Q).
\]
Then $\widetilde{\dot K_j}=\dot m_j$ and $\dot m_j$ is supported in
$[-1/4,1/4]^d$. Set $g^Q(x)=g(x/Q)$. A change of variables gives
\[
 (\dot K_j\bigstar g)(x)=(K_j\bigstar g^Q)(Qx),
\]
which leads to
\begin{align*}
 \left\|\sup_j|\dot K_j\bigstar g|\right\|_{L^{p,\infty}}
 &=Q^{-d/p}\left\|\sup_j|K_j\bigstar g^Q|\right\|_{L^{p,\infty}}\\
 &\le Q^{-d/p}W\|g^Q\|_{L^{p,1}}
 =W\|g\|_{L^{p,1}}.
\end{align*}

Let $\dot T_j$ be the discrete operator whose kernel is
$\dot K_j|_{\Z^d}$. For $b\in\Z_Q^d$, define $h_b(n)=h(b+Qn)$.
Then, by \eqref{dilationKQ}, 
\[
 T_j^Qh(b+Qn)
 =\sum_{m\in\Z^d}Q^dK_j(Qm)h_b(n-m)
 =\sum_{m\in\Z^d}\dot K_j(m)h_b(n-m)
 =\dot T_jh_b(n),\quad n\in\Z^d.
\]
Apply the $Q=1$ estimate to $\dot m_j$ and $h_b$. For each $s>0$, we have
\begin{align*}
 \#\left\{x\in\Z^d:\sup_j|T_j^Qh(x)|>s\right\}
 &=\sum_{b\in\Z_{Q}^d}\#\left\{n\in\Z^d:\sup_j|\dot T_jh_b(n)|>s\right\}\\
 &\le\left(\frac{C_{d,p}W}{s}\right)^p
       \sum_{b\in\Z_{Q}^d}\|h_b\|_{\ell^{p,1}}^p.
\end{align*}
For the remaining sum, Minkowski's integral inequality gives
\begin{align*}
 \left(\sum_{b\in\Z_Q^d}\|h_b\|_{\ell^{p,1}}^p\right)^{1/p}
 &\le\int_0^\infty
       \left(\sum_{b\in\Z_Q^d}\#\{n:|h_b(n)|>s\}\right)^{1/p}\,ds\\
 &=\|h\|_{\ell^{p,1}}.
\end{align*}
For $p=1$ this is an equality. Thus taking the supremum over $s$
proves \eqref{eq:weak-sampling}, without a loss from the number of cosets.

For general $h\in\ell^{p,1}$, the truncations
$h_R=h\one_{[-R,R]^d}$ converge in $\ell^{p,1}$.
Each fixed discrete kernel is rapidly decreasing and
$h\in\ell^{p,1}\subset\ell^\infty$, so they converge at every
lattice point. Fatou's lemma for the level sets extends the estimate
to $h$. If a countable multiplier family has the same continuous
bound for every finite $J$, increasing finite subsets and continuity
from below give the corresponding estimate for its full supremum.
The constants remain independent of $Q$ and $J$.
\end{proof}
\medskip

\section{An effective MSW decomposition}\label{sec:3}
\subsection{A major-arc model with a common modulus}
In this section, we refine the major-arc decomposition in
\cite{MR4041278} for factorial squared radii by constructing a model
with a common modulus and estimating its $\ell^2$ approximation error
on the factorial tail. We factor the model into arithmetic and analytic
parts and use continuous maximal estimates and sampling to obtain
quantitative strong bounds. We then establish a separate decomposition
for the logarithmic endpoint estimates.

Fix an integer $N\geq2$ and set $Q=N!$. Define the major-arc multiplier with modulus $Q$ by
\begin{equation}\label{eq:b-lambda}  b_{\lambda,N}(\xi)=\cMSW\sum_{0\leq a'<Q}\sum_{{\ell'\in\Z^d}}
G(a',\ell',Q)\,\widetilde\psi_{2Q}(\xi-\ell'/Q)\,
\widetilde{d\sigma_\lambda}(\xi-\ell'/Q),\quad\lambda>0,
\end{equation}
where the periodic sum is defined on the torus $\mathbb T^d$.
Let $B_{\lambda,N}$ denote the corresponding convolution operator.
Factorial divisibility will be used only when this model is compared
with the low major arcs on the tail $\lambda_k=\mu_k!$,
$\mu_k\geq N^3$.
 We first decompose the major-arc multiplier $m_\lambda$ into
the high part $m_{\lambda,>N}$ and the low part $m_{\lambda,\leq N}$.  We next decompose the low part $m_{\lambda,\leq N}$ into $b_{\lambda,N}$ and $m_{\lambda,\leq N}-b_{\lambda,N}$. Recall that $m_\lambda(\xi)$ is given by
\begin{align*}
   \cMSW\sum_{q=1}^{\infty} \sum_{a \in\Z_q^\times} e_q(-\lambda a) m^{a/q}_\lambda(\xi).
\end{align*}
Define the low part of $m_\lambda$ by
\begin{align}\label{low}
   m_{\lambda,\leq N}(\xi):= \cMSW\sum_{q=1}^{N} \sum_{a \in \Z_q^\times} e_q(-\lambda a) m^{a/q}_\lambda(\xi)
\end{align}
and the high part of $m_\lambda$ by
\begin{align}\label{high}
     m_{\lambda,> N}(\xi):=\cMSW\sum_{q=N+1}^{\infty} \sum_{a \in \Z_q^\times} e_q(-\lambda a) m^{a/q}_\lambda(\xi).
\end{align}
Thus $m_\lambda=m_{\lambda,\leq N}+m_{\lambda,> N}$.
Adding and subtracting $b_{\lambda,N}$ gives
\begin{align}\label{highlowdecom1}
    m_\lambda= b_{\lambda,N}+\left[m_{\lambda,\leq N}-b_{\lambda,N}\right]+m_{\lambda,> N}.
\end{align}

\subsubsection{Equivalent form of $b_{\lambda,N}$}
To simplify the term $m_{\lambda,\leq N}-b_{\lambda,N}$, we compare the $Q$-model with the reduced arcs of denominator at most $N$.  We first reduce the $Q$-Gauss sums.  For $a'\in\Z_Q$ and $\ell'\in\Z_Q^d$, write
\[
\rho=\gcd(a',\ell'_1,\ldots,\ell'_d,Q).
\]
The normalized Gauss sum satisfies
\begin{align*}
  G(a',\ell',Q)=
  \begin{cases}
 G(a'/\rho,\ell'/\rho,Q/\rho),& \gcd(a',Q)=\rho,\\
  0,&\gcd(a',Q)>\rho.
  \end{cases}
\end{align*}
In fact, if $\gcd(a',Q)$ does not divide $\ell'_j$ for some $j$, then summing first in the $j$-th variable gives zero. Otherwise, $\gcd(a',Q)|\ell'_j$ for all $j$, and factoring out the common divisor reduces the normalized Gauss sum from modulus $Q$ to modulus $Q/\gcd(a',Q)$. This gives the above equality in the case that $\gcd(a',Q)=\rho$.

Consequently, the multiplier \eqref{eq:b-lambda} can be written, 
for every $\lambda>0$, as
\begin{equation}\label{eq:B-grouped}
  b_{\lambda,N}(\xi)=\cMSW
  \sum_{q\mid Q}\sum_{a\in\Z_q^\times}\sum_{{\ell\in\Z^d}}
  G(a,\ell,q)\,\widetilde\psi_{2Q}(\xi-\ell/q)\,
  \widetilde{d\sigma_\lambda}(\xi-\ell/q).
\end{equation}
Here, $q=Q/\rho$, $a=a'/\rho$ and $\ell=\ell'/\rho$.  Conversely, every triple
$q\mid Q$, $a\in\Z_q^\times$, $\ell\in\Z^d$ occurs once.
This is an algebraic identity for every $\lambda>0$; no divisibility
condition on $\lambda$ has been used. 

For each $q|Q$ and $a\in\Z_q^\times$, define the multiplier
\begin{align}
    b_{\lambda,N}^{a/q}(\xi):=\sum_{\ell\in\Z^d}G(a,\ell,q)\,\widetilde\psi_{2Q}(\xi-\ell/q)\,
  \widetilde{d\sigma_\lambda}(\xi-\ell/q).
\end{align}
Then, from (\ref{low}) and (\ref{eq:B-grouped}), one has 
\begin{align}\label{highlowdecom2}
  m_{\lambda,\leq N}(\xi)-b_{\lambda,N}(\xi)
  &=-\cMSW
  \sum_{\substack{q\mid Q\\q>N}}
  \sum_{a\in(\Z/q\Z)^\times}b^{a/q}_{\lambda,N}(\xi)\notag\\
  &\quad+\cMSW
  \sum_{q\leq N}\sum_{a\in(\Z/q\Z)^\times}
  \left[e_q(-\lambda a)m^{a/q}_\lambda(\xi)-b^{a/q}_{\lambda,N}(\xi)\right].
\end{align}

\subsection{Tail approximation and maximal bounds}
We prove the $\ell^2$ estimate for $\sup_{k:\,\mu_k\geq N^3}
| \overline A_{\lambda_k}f-B_{\lambda_k,N}f|$ in this subsection. Our argument relies on the $\ell^2$ estimate for $E_\lambda$ in Proposition~\ref{minor}.

\begin{lemma}\label{lem:tail-error}
Let $d\geq5$ and let $\lambda_k=\mu_k!$, where $\{\mu_k\}$ is a strictly increasing sequence of positive integers. For every integer $N\geq2$,
\begin{equation}\label{eq:tail-error}
  \left\|\sup_{k:\,\mu_k\geq N^3}| 
  \overline A_{\lambda_k}f-B_{\lambda_k,N}f|\right\|_{\ell^2}
  \lesssim_d N^{-(d-4)/2}\|f\|_{\ell^2}.
\end{equation}
\end{lemma}
\begin{proof}
By the triangle inequality, (\ref{eq:MSW_decomp}) and (\ref{highlowdecom1}), one has, for all $k>0$
\begin{align}\label{tri}
    \left| \overline A_{\lambda_k} f-B_{\lambda_k,N}f\right|\leq  \left|E_{\lambda_k} f\right|+  \left|M_{\lambda_k,>N} f\right|+\left|M_{\lambda_k,\leq N} f-B_{\lambda_k,N} f\right|.
\end{align}

We will estimate the $\ell^2$-norm of each part, starting with the error $E_{\lambda_k}$.\\
  Set $\Lambda_N=(N^3)!$. Then $\mu_k\geq N^3$ implies
$\lambda_k\geq\Lambda_N$. A dyadic decomposition and
Proposition~\ref{minor} give
\begin{align*}
  \left\|\sup_{k:\,\mu_k\geq N^3}|E_{\lambda_k}f|\right\|_{\ell^2}
  &\leq \sum_{j\geq0}
  \left\|\sup_{2^j\Lambda_N\leq \lambda<2^{j+1}\Lambda_N}|E_\lambda f|\right\|_{\ell^2} 
  \lesssim _d \Lambda_N^{-(d-4)/4}\|f\|_{\ell^2}.
\end{align*}
As a result, we have 
\begin{align}
    \left\|\sup_{k:\,\mu_k\geq N^3}|E_{\lambda_k}f|\right\|_{\ell^2}\lesssim_d  N^{-(d-4)/2}\|f\|_{\ell^2},
  \label{eq:tail-minor-detail}
\end{align}
where we used the elementary bound $\Lambda_N\geq N^2$.

Next we bound the term $M_{\lambda_k,>N}f$. From the definition of $M_{\lambda_k,>N}f$ and the maximal function estimate (\ref{eq:single-denom-l2}), we have
\begin{align}
  \left\|\sup_{k:\,\mu_k\geq N^3}
  \left|M_{\lambda_k,>N}f\right|\right\|_{\ell^2}&\lesssim_d\left\|\sup_{k:\,\mu_k\geq N^3}
  \left|\sum_{q>N}\sum_{a\in(\Z/q\Z)^\times}e_q(-\lambda_ka)M_{\lambda_k}^{a/q}f\right|\right\|_{\ell^2} \notag\\
  &\leq
  \sum_{q>N}\sum_{a\in(\Z/q\Z)^\times}
  \left\|M^{a/q}_\star f\right\|_{\ell^2}\notag\\
  &\lesssim_d \sum_{q>N}\sum_{a\in(\Z/q\Z)^\times} q^{-d/2}\|f\|_{\ell^2}.
  \label{eq:tail-large-q-detail}
\end{align}
There are $\varphi(q) \le q$ primitive residues
$a \in (\Z/q\Z)^\times$, where $\varphi$ is Euler's totient function.
Summing \eqref{eq:tail-large-q-detail} over these residues gives
\begin{align}\label{FFFF}
    \left\|\sup_{k:\,\mu_k\geq N^3}
  \left|M_{\lambda_k,>N}f\right|\right\|_{\ell^2}\lesssim_d \sum_{q>N}q^{1-d/2}\|f\|_{\ell^2}
 \simeq N^{-(d-4)/2}\|f\|_{\ell^2} .
\end{align}

It remains to estimate the term 
$\left|M_{\lambda_k,\leq N} f-B_{\lambda_k,N} f\right|$. Note that for all $\mu_k\geq N^3$ and $q\leq N$, $\lambda_k=\mu_k!$ is divisible by $q$, and hence, in this case,  $e_q(-\lambda_k a)=e(-\lambda_k a/q)=1$. Therefore, from (\ref{highlowdecom2}),  the multiplier of $M_{\lambda_k,\leq N} -B_{\lambda_k,N}$ can be written as
\begin{align*}
m_{\lambda_k,\leq N}(\xi)-b_{\lambda_k,N}(\xi)
&=-\cMSW\sum_{\substack{q\mid Q\\q>N}}
  \sum_{a\in(\Z/q\Z)^\times}b^{a/q}_{\lambda_k,N}(\xi)\\
&\quad+\cMSW\sum_{q\leq N}\sum_{a\in(\Z/q\Z)^\times}
  \left[m^{a/q}_{\lambda_k}(\xi)-b^{a/q}_{\lambda_k,N}(\xi)\right].
\end{align*}
Equivalently, for each $\mu_k\geq N^3$
\begin{align}\label{cutoffdiff}
M_{\lambda_k,\leq N}-B_{\lambda_k,N}
&=-\cMSW\sum_{\substack{q\mid Q\\q>N}}
  \sum_{a\in(\Z/q\Z)^\times}B^{a/q}_{\lambda_k,N}\notag\\
&\quad+\cMSW\sum_{q\leq N}\sum_{a\in(\Z/q\Z)^\times}
  \left[M^{a/q}_{\lambda_k}-B^{a/q}_{\lambda_k,N}\right].
\end{align}
For the first term, the proof of \eqref{eq:single-denom-l2} in
\cite[proof of Proposition~3.1(a), pp.~199--201]{MR1888798} also applies
with $\widetilde\psi_{2Q}$ in place of $\widetilde\psi_q$.
Since $q\mid Q$, the support of $\widetilde\psi_{2Q}$ lies in
$(-1/(2q),1/(2q)]^d$. The Gauss multiplier bound is unchanged,
and dilation preserves the $L^1$ norm of the smoothing kernel.
Thus the sampling argument gives the same $\ell^2$ bound uniformly
in $Q$. Summing as in \eqref{eq:tail-large-q-detail} yields
\begin{align}
  \left\|\sup_{k:\,\mu_k\geq N^3}
|\sum_{\substack{q\mid Q\\ q>N}}\sum_{a\in(\Z/q\Z)^\times}B_{\lambda_k,N}^{a/q}f|\right\|_{\ell^2}
  &\lesssim_d N^{-(d-4)/2}\|f\|_{\ell^2} .
  \label{eq:B-large-q-detail}
\end{align}

We finally treat the second term on the right of (\ref{cutoffdiff}).
By definition, the multiplier of $M^{a/q}_{\lambda_k}-B^{a/q}_{\lambda_k,N}$ is
\begin{equation}\label{eq:cutoff-diff-mult}
  \left[ m^{a/q}_{\lambda_k}(\xi)-b^{a/q}_{\lambda_k,N}(\xi)\right]=
  \sum_{\ell\in\Z^d}G(a,\ell,q)
  \bigl[\widetilde\psi_q(\xi-\ell/q)-\widetilde\psi_{2Q}(\xi-\ell/q)\bigr]
  \widetilde{d\sigma_{\lambda_k}}(\xi-\ell/q).
\end{equation}
The bracket in \eqref{eq:cutoff-diff-mult} vanishes when
$|\xi-\ell/q|\leq1/(8Q)$ and is supported where
$|\xi-\ell/q|\leq1/(2q)$.   Combining the Gauss-sum bound $|G(a,\ell,q)|\lesssim_d q^{-d/2}$
with the Fourier decay of the spherical measure \cite{MR1232192} gives
\begin{equation}\label{eq:stationary-phase-use}
  \left|m^{a/q}_{\lambda_k}(\xi)-b^{a/q}_{\lambda_k,N}(\xi)\right|
  \lesssim_d q^{-d/2}
  \sum_{\ell\in\Z^d}\mathbf 1_{c/Q\leq |\xi-\ell/q|\leq C/q}
  (1+\lambda_k^{1/2}|\xi-\ell/q|)^{-(d-1)/2}.
\end{equation}
If $\mu_k\geq N^3$, then
$\lambda^{1/2}_{k+1}/\lambda^{1/2}_k\geq (\mu_k+1)^{1/2}\geq2$.  Hence, for every $\eta$ with $|\eta|\geq c/Q$, 
\[
  \sum_{k:\,\mu_k\geq N^3}(1+\lambda^{1/2}_k|\eta|)^{-(d-1)}
  \lesssim_d (\Lambda^{1/2}_N|\eta|)^{-(d-1)}
  \lesssim_d (\Lambda^{1/2}_N/Q)^{-(d-1)}.
\]
At each $\xi\in\T^d$, the translated cutoff supports in
\eqref{eq:cutoff-diff-mult} have bounded overlap.  Squaring
\eqref{eq:stationary-phase-use}, summing in $k$, and using the preceding
geometric-series estimate therefore gives
\begin{align*}
  \sup_{\xi\in\T^d}
  \sum_{k:\,\mu_k\geq N^3}
  \left|m^{a/q}_{\lambda_k}(\xi)-b^{a/q}_{\lambda_k,N}(\xi)\right|^2
  &\lesssim_d q^{-d}\sup_{\xi\in\T^d}
 \sum_{\ell\in\Z^d}\mathbf 1_{c/Q\leq |\xi-\ell/q|\leq C/q}
 \cdot(\Lambda^{1/2}_N/Q)^{-(d-1)}\\&\lesssim_d q^{-d}(\Lambda_N^{1/2}/Q)^{-(d-1)}.
\end{align*}
Thus $\sup_k|u_k|\leq(\sum_k|u_k|^2)^{1/2}$ and Plancherel's theorem yield,
for each fixed $q\leq N$ and $a\in\Z_q^\times$,
\begin{equation}\label{eq:cutoff-diff-L2}
  \left\|\sup_{k:\,\mu_k\geq N^3}
  |(M^{a/q}_{\lambda_k}-B^{a/q}_{\lambda_k,N})f|\right\|_{\ell^2}
  \lesssim_d q^{-d/2}(\Lambda^{1/2}_N/Q)^{-(d-1)/2}\|f\|_{\ell^2} .
\end{equation}
Summing \eqref{eq:cutoff-diff-L2} over $q\leq N$ and $a\in\Z_q^\times$ gives
\begin{align}
\left\|\sup_{k:\,\mu_k\geq N^3}
  |\sum_{q\leq N}\sum_{a\in\Z_q^\times}
 (M^{a/q}_{\lambda_k}-B^{a/q}_{\lambda_k,N})f|\right\|_{\ell^2} &\lesssim_d (\Lambda^{1/2}_N/Q)^{-(d-1)/2}
  \sum_{q\leq N}q^{1-d/2}\|f\|_{\ell^2}\notag\\
 &\lesssim_d (\Lambda^{1/2}_N/Q)^{-(d-1)/2}\|f\|_{\ell^2},
  \label{eq:small-q-cutoff-detail}
\end{align}
where the last inequality follows from the restriction $d\geq5$.
For the final comparison, $\Lambda_N\geq QN^{N^3-N}$ and
$Q\leq N^N$ give
\[
 \frac{Q^2}{\Lambda_N}\leq N^{2N-N^3}
 \quad\text{and}\quad 
 (\Lambda_N^{1/2}/Q)^{-(d-1)/2}
 \leq N^{-(N^3-2N)(d-1)/4}
 \leq N^{-(d-4)/2}.
\]
The last inequality holds for $N\geq2$ and $d\geq5$. Since the
exponent $(N^3-2N)(d-1)/4$ tends to infinity, the same factor is
$O_{d,A}(N^{-A})$ for every fixed $A>0$.

Combining \eqref{tri}, \eqref{eq:tail-minor-detail}, \eqref{FFFF}, \eqref{eq:B-large-q-detail}, and \eqref{eq:small-q-cutoff-detail}, we obtain \eqref{eq:tail-error}.
\end{proof}

\subsubsection{Factorization and the bound for $U_Q$}
The factorization of $B_{\lambda,N}$ originates in  \cite[Section 3]{MR1888798}. See also \cite{MR3982252} for the application of this factorization.
We factorize the convolution operator $B_{\lambda,N}$ as
\begin{equation}\label{eq:factor}
  B_{\lambda,N}=T_{\lambda,Q}\circ U_Q,
\end{equation}
where $T_{\lambda,Q}$ is the convolution operator with multiplier
\begin{align}
  t_{\lambda,Q}(\xi)=\sum_{\ell'\in\Z^d}
  \widetilde\psi_{2Q}(\xi-\ell'/Q)\,
  \widetilde{d\sigma_\lambda}(\xi-\ell'/Q),\label{eq:t-R}
\end{align}
and $U_Q$ is the convolution operator with multiplier
\begin{align}
  u_Q(\xi)=\cMSW\sum_{\ell'\in\Z^d}\sum_{0\leq a'<Q}G(a',\ell',Q)
  \widetilde\psi_Q(\xi-\ell'/Q).\label{eq:u-N}
\end{align}
Since $\widetilde\psi_Q=1$ on the support of $\widetilde\psi_{2Q}$,
we have $\widetilde{\psi}_{2Q}=\widetilde{\psi}_{2Q}\cdot\widetilde{\psi}_Q$.
Distinct centers in $Q^{-1}\mathbb Z^d/\mathbb Z^d$ are separated
by at least $1/Q$, whereas the two cutoff support radii sum to $3/(4Q)$.
Thus the product has no cross terms, and
$b_{\lambda,N}(\xi)=t_{\lambda,Q}(\xi)\cdot u_Q(\xi)$.
This shows that \eqref{eq:factor} holds.

We first obtain a uniform $\ell^1$ bound for the kernel of $U_Q$
and apply Young's inequality. We then use continuous maximal estimates
and sampling to control the maximal operator associated with
$T_{\lambda_k,Q}$ on the factorial tail.

\begin{lemma}\label{lem:U-bound}
For $1\leq p\leq2$,
\begin{equation}\label{eq:U-bound}
  \|U_Q f\|_{\ell^p}\lesssim_d \|f\|_{\ell^p},
\end{equation}
with a constant independent of $p,N$, and $Q$.
\end{lemma}
\begin{proof}
To apply Young's convolution inequality, it suffices to obtain a uniform
$\ell^1$ bound for the convolution kernel of $U_Q$.
Let $K^U_Q$ denote the kernel of $U_Q$ on $\Z^d$.
Then, for each $m\in\Z^d$, we have
\begin{align*}
  K^U_Q(m)&=\cMSW\sum_{\ell'\in\Z^d}\sum_{0\leq a'<Q}G(a',\ell',Q)e_Q(m\ell')\int_{\mathbb T^d-\ell'/Q}
  \widetilde\psi_Q(\xi)e(m\xi)\,d\xi\\
&=\cMSW\mathcal{F}^{-1}\left[{\widetilde\psi_Q}\right](m)\cdot Q\,\mathbf 1_{\{|m|^2\equiv0\pmod Q\}},
\end{align*}
where $\mathcal{F}^{-1}\left[{\widetilde\psi_Q}\right]$ is the inverse Fourier transform of ${\widetilde\psi_Q}$ on $\R^d$. Indeed, from the definition of $  G(a',\ell',Q)$ and the orthogonality of characters on \(\Z_Q^d\),
we have
\begin{align*}
\sum_{0\le a'<Q}\sum_{r\in\Z_Q^d}
  G(a',r,Q)e_Q(m\cdot r)&=
Q^{-d}
\sum_{0\le a'<Q}\sum_{r\in\Z_Q^d}\sum_{n\in\Z_Q^d}
e_Q(a'|n|^2)
e_Q((n+m)\cdot r)\\
&=Q^{1-d}
\sum_{n\in\Z_Q^d}\sum_{r\in\Z_Q^d}e_Q((n+m)\cdot r)
\cdot \mathbf 1_{\{|n|^2\equiv 0\pmod Q\}}\\
&=
Q\,\mathbf 1_{\{|m|^2\equiv0\pmod Q\}}.
\end{align*}

Since the cutoff function ${\widetilde\psi_Q}$ is smooth at frequency scale $Q^{-1}$, for every $M>d$,
\begin{align*}\left|\mathcal{F}^{-1}\left[{\widetilde\psi_Q}\right](m)\right|
  \lesssim_M Q^{-d}\left(1+\frac{|m|}{Q}\right)^{-M}.
\end{align*}
Therefore, for each $M>d$
\begin{equation}\label{eq:UN-kernel-bound}
  \|K^U_Q\|_{\ell^1}
  \lesssim_M Q^{1-d}
  \sum_{m\in\Z^d}\left(1+\frac{|m|}{Q}\right)^{-M}
  \mathbf 1_{\{|m|^2\equiv0\pmod Q\}}.
\end{equation}
For $s\geq0$, the number of $m$ with $|m|\leq 2^{s+1}Q$ and $Q\mid |m|^2$ is bounded by
\[
  1+\sum_{1\leq j\leq 2^{2s+2}Q} r_d(jQ)
  \lesssim_d 1+\sum_{1\leq j\leq2^{2s+2}Q}(jQ)^{d/2-1}
  \lesssim_d 2^{sd}Q^{d-1},
\]
where $r_d(n)=\#\{m\in\Z^d:|m|^2=n\}$ and we used the standard bound
$r_d(n)\lesssim_d n^{d/2-1}$ for $d\geq5$; see, for example,
\cite{MR0097357} and \cite[p.~190]{MR1888798}.  Splitting the sum in \eqref{eq:UN-kernel-bound} into 
the ball $|m|\leq Q$ and the annuli
$2^sQ<|m|\leq2^{s+1}Q$ gives
\begin{align}\label{L^1normandcounting}
  \|K^U_Q\|_{\ell^1}
  \lesssim_{d,M}1+\sum_{s\geq0}2^{-sM}2^{sd}
  \lesssim_d1,
\end{align}
after choosing $M>d$.  Young's inequality gives \eqref{eq:U-bound} for every $1\leq p\leq\infty$, in particular for $1\leq p\leq2$.
\end{proof}

\subsubsection{A quantitative superlacunary continuous estimate}
Let $\MHL$ denote the Hardy--Littlewood maximal operator.  For $r>0$, $d\sigma_r$ denotes normalized Euclidean surface measure on
$\{x\in\R^d:|x|=r^{1/2}\}$.  If
$\mathcal R=\{r_j:j\in J\}$ is a finite increasing family, write
\[
  \mathcal M_{\mathcal R}F(x)
  :=\sup_{j\in J}|d\sigma_{r^2_j}\bigstar F(x)|.
\]
The proof of the main theorem in this section is motivated by \cite{MR0537803}.
\begin{theorem}
\label{thm:superlac}~
Let $d\geq5$, let $L\geq1$ be an integer, and suppose that
\begin{align}\label{cond:superlac}
r_{j+1}/r_j\geq2^L
\end{align}
for consecutive radii in $\mathcal R$.  Then, for every $1<p\leq2$,
\begin{equation}\label{eq:superlac-main}
  \|\mathcal M_{\mathcal R}F\|_{L^p(\R^d)}
  \lesssim_d p'
  \left(1+\frac{1}{((p-1)L)^2}\right)
  \|F\|_{L^p(\R^d)},
  \quad\text{where}\,\, p'=\frac{p}{p-1}.
\end{equation}
The constant is independent of $L$, of $\mathcal R$, and of its
cardinality. 
\end{theorem}

\begin{proof}
We prove (\ref{eq:superlac-main}) for  $F\in\mathcal S(\R^d)$.  Fix nonnegative functions
$\omega,\Phi\in C_c^\infty(\R^d)$ such that
\[
  \int_{\R^d}\omega(x)\dd x
  =\int_{\R^d}\Phi(x)\dd x=1
\]
and write $h_s(x)=s^{-d}h(x/s)$ for all $s>0$.  Define
\begin{align*}
    \tau=d\sigma_1-\omega(x)\dd x,
  \quad\text{and}\quad
  \tau_r=d\sigma_{r^2}-\omega_r(x)\dd x.
\end{align*}
By definition, $\tau(\R^d)=0$ and $\tau_r$ is the pushforward of $\tau$ under
$Dil_r:x\mapsto rx$.

For $b\geq0$, let $ K_{r,b}$ be the kernel given by
$$
  K_{r,b}
  :=\bigl(\Phi_{r2^{-(b+1)L}}-\Phi_{r2^{-bL}}\bigr)\bigstar\tau_r.
$$
Since $\Phi_s$ is an approximate identity as $s\downarrow0$, telescoping in
$\mathcal S'(\R^d)$ gives
\begin{align}\label{eq:superlac-telescope}
  d\sigma_{r^2}
  =\left[\omega_r(x)\dd x+\Phi_r\bigstar\tau_r\right]
   +\sum_{b=0}^{\infty}K_{r,b}:=L_r +\sum_{b=0}^{\infty}K_{r,b},
\end{align}
where $L_r:=\omega_r+\Phi_r\bigstar\tau_r$ satisfies
$|L_r(x)|\lesssim_d r^{-d}\mathbf1_{\{|x|\leq C_dr\}}$ for some $C_d>0$.
Consequently, the maximal operator $ \sup_{r>0}|L_r\bigstar F|$ is dominated by the Hardy--Littlewood maximal function and hence
\begin{equation}\label{eq:superlac-low-HL}
  \left\|\sup_{r>0}|L_r\bigstar F|\right\|_{L^p(\R^d)}
  \lesssim_d p'\|F\|_{L^p(\R^d)}.
\end{equation}
\smallskip

To estimate $\sup_{j\in J}|K_{r_j,b}\bigstar F|$, we study the
randomized convolution operators defined below. We first establish
an $L^2$ bound and a weak type $(1,1)$ bound for these operators,
then interpolate and apply Khintchine's inequality.

For signs $\varepsilon_j\in\{-1,1\}$, define
$
  T_{b,\varepsilon}F
  :=\sum_{j\in J}\varepsilon_jK_{r_j,b}\bigstar F.$
We claim that the following $L^2$ estimate for $ T_{b,\varepsilon}F$ is uniform in the signs and in $J$:
\begin{equation}\label{eq:superlac-block-L2}
  \|T_{b,\varepsilon}F\|_2
  \lesssim_d2^{-bL}\|F\|_2.
\end{equation}
 To derive the claim, write
\[
  \Theta_{b,L}(\zeta)
  :=\widetilde\Phi(2^{-(b+1)L}\zeta)
    -\widetilde\Phi(2^{-bL}\zeta).
\]
Let $m_{j,b}$ be the multiplier of $K_{r_j,b}$. Then
$m_{j,b}(\xi)=\Theta_{b,L}(r_j\xi)\cdot \widetilde\tau(r_j\xi)$.
The cancellation property of $\tau$ and the stationary phase method imply, since $d\geq5$, that
\begin{equation}\label{eq:superlac-tau-fourier}
  |\widetilde\tau(\zeta)|
  \lesssim_d\min\bigl\{|\zeta|,(1+|\zeta|)^{-2}\bigr\}.
\end{equation}
Besides, for every $M_0\geq1$, smoothness and rapid decay of $\widetilde\Phi$ give
\begin{equation}\label{eq:superlac-Theta-bounds}
  |\Theta_{b,L}(\zeta)|
  \lesssim_{M_0}\min\left\{
      2^{-bL}|\zeta|,\ 1,\
      \bigl(2^{-(b+1)L}|\zeta|\bigr)^{-M_0}
  \right\}.
\end{equation}
Fix $\xi\neq0$ and put $u_j=r_j|\xi|$.
By assumption, the sequence $\{u_j\}$ is $2^L$-lacunary.
On the one hand, \eqref{eq:superlac-tau-fourier} and
\eqref{eq:superlac-Theta-bounds} give
\begin{align*}
  \sum_{u_j\leq1}|m_{j,b}(\xi)|
  \lesssim_d2^{-bL}\sum_{u_j\leq1}u_j^2
   \lesssim_d2^{-bL},
\end{align*}
and
\begin{align*}
    \sum_{1<u_j\leq2^{bL}}|m_{j,b}(\xi)|
  &\lesssim_d2^{-bL}\sum_{1<u_j\leq2^{bL}}u_j^{-1}
   \lesssim_d2^{-bL};
\end{align*}
on the other hand, there are at most two indices for which
$2^{bL}<u_j\leq2^{(b+1)L}$, 
hence
\begin{align*}
  &\sum_{2^{bL}\leq u_j\leq2^{(b+1)L}}|m_{j,b}(\xi)|+\sum_{u_j>2^{(b+1)L}}|m_{j,b}(\xi)|\\
  &\lesssim_d 2^{-2bL}+ 2^{-2(b+1)L}
  \sum_{u_j>2^{(b+1)L}}
  \left(\frac{u_j}{2^{(b+1)L}}\right)^{-4}\lesssim_d2^{-2bL}.
\end{align*}
Thus
$$
  \sup_{\xi\in\R^d}\sum_{j\in J}|m_{j,b}(\xi)|
  \lesssim_d2^{-bL},
$$
and therefore \eqref{eq:superlac-block-L2} follows from Plancherel's theorem.
\smallskip

We next prove the weak type $(1,1)$ estimate
\begin{equation}\label{eq:superlac-block-weak11}
  \|T_{b,\varepsilon}F\|_{L^{1,\infty}}
  \lesssim_d(1+b)\|F\|_1.
\end{equation}
Young's inequality and scaling give
\begin{equation}\label{eq:superlac-kernel-basic}
  \|K_{r,b}\|_1\lesssim_d1,
  \qquad
  \|\nabla K_{r,b}\|_1\lesssim_d\frac{2^{(b+1)L}}{r}
  \quad\text{and}\quad
  \supp K_{r,b}\subset B(0,C_dr),
\end{equation}
which implies, for every $y\in\R^d$,
$$
  \int_{\R^d}|K_{r,b}(x-y)-K_{r,b}(x)|\dd x
  \lesssim_d\min\left\{1,\frac{2^{(b+1)L}|y|}{r}\right\}.
$$
Let $K_{b,\varepsilon}=\sum_j\varepsilon_jK_{r_j,b}$.  If $C_0$ is
sufficiently large, the support estimate in
\eqref{eq:superlac-kernel-basic} shows that the terms with
$r_j\leq c_d|y|$ vanish in the integral over $|x|>C_0|y|$.  Therefore,
\begin{align*}
 \int_{|x|>C_0|y|}
  |K_{b,\varepsilon}(x-y)-K_{b,\varepsilon}(x)|\dd x\lesssim_d
  \sum_{r_j>c_d|y|}
  \min\left\{1,\frac{2^{(b+1)L}|y|}{r_j}\right\}.
  \label{eq:superlac-Hormander-sum}
\end{align*}
From the hypothesis (\ref{cond:superlac}), there are $O_d(1+b)$ radii in the intermediate range
$c_d|y|<r_j\leq2^{(b+1)L}|y|$, while the remaining terms form a geometric
series.  Consequently,
\begin{equation}\label{eq:superlac-Hormander}
  \sup_{y\in\R^d}
  \int_{|x|>C_0|y|}
  |K_{b,\varepsilon}(x-y)-K_{b,\varepsilon}(x)|\dd x
  \lesssim_d1+b.
\end{equation}
With (\ref{eq:superlac-Hormander}) in hand, the assertion
(\ref{eq:superlac-block-weak11}) follows from the Calder\'on--Zygmund theory.
For completeness, apply the Calder\'on--Zygmund decomposition
to $F$ at height $\alpha$ and write
\[
  F=g+\sum_{I\in\mathcal I}b_I,
\]
where $\|g\|_2^2\lesssim_d\alpha\|F\|_1$,
$\sum_{I\in\mathcal I}|I|\lesssim_d\alpha^{-1}\|F\|_1$,
$\int b_I=0$, and
$\sum_I\|b_I\|_1\lesssim_d\|F\|_1$.  Since
$2^{-bL}\leq1$, \eqref{eq:superlac-block-L2} and Chebyshev's inequality
give
\[
  |\{|T_{b,\varepsilon}g|>\alpha/2\}|
  \lesssim_d\alpha^{-2}\|g\|_2^2
  \lesssim_d\alpha^{-1}\|F\|_1.
\]
Let $I^*$ be a fixed sufficiently large concentric dilate of $I$.
Then the measure of $\bigcup_I I^*$ is $O_d(\alpha^{-1}\|F\|_1)$.  If $c_I$
denotes the center of $I$, then, for $x\notin I^*$, the cancellation property of $b_I$ gives
\[
  T_{b,\varepsilon}b_I(x)
  =\int_I\bigl[K_{b,\varepsilon}(x-y)
       -K_{b,\varepsilon}(x-c_I)\bigr]b_I(y)\dd y.
\]
After translating $x$ by $c_I$, the region $x\notin I^*$ lies in
$|x-c_I|>C_0|y-c_I|$.  Hence \eqref{eq:superlac-Hormander}, Fubini's
theorem, and the preceding bounds for the bad functions imply
\[
  \int_{(\cup_I I^*)^c}
  \left|T_{b,\varepsilon}\left(\sum_Ib_I\right)(x)\right|\dd x
  \lesssim_d(1+b)\sum_I\|b_I\|_1
  \lesssim_d(1+b)\|F\|_1.
\]
The application of Chebyshev's inequality now proves \eqref{eq:superlac-block-weak11}, as desired.
\smallskip

We now derive the $L^p$ bound for $T_{b,\varepsilon}$ from
\eqref{eq:superlac-block-L2} and
\eqref{eq:superlac-block-weak11}. To track the dependence on $p$
throughout $1<p\leq2$, we use Marcinkiewicz interpolation for
$1<p\leq3/2$ and Riesz--Thorin interpolation for $3/2<p\leq2$.
For $1<p\leq2$, set $\theta_p=2(p-1)/p$, so that $$1/p=(1-\theta_p)+\theta_p/2.$$
For $1<p\leq3/2$, Marcinkiewicz interpolation \cite[Chapter~1]{MR2445437} between
\eqref{eq:superlac-block-weak11} and \eqref{eq:superlac-block-L2} gives
\begin{align*}
  \|T_{b,\varepsilon}\|_{L^p\to L^p}&\lesssim\left[p'+p/(2-p)\right]\cdot(1+b)^{1-\theta_p}2^{-bL\theta_p}\\
  &\lesssim_d p'(1+b)^{1-\theta_p}2^{-bL\theta_p}\quad\forall\,\,1<p\leq 3/2;
\end{align*}
here the interpolation constant is $O(p')$ because $2-p\geq1/2$ in
this range.\\
For $3/2<p\leq2$, put $s=4-6/p$, so that
$1/p=(1-s)/(3/2)+s/2$. We apply Riesz--Thorin
interpolation between the estimate 
\[
  \|T_{b,\varepsilon}\|_{L^{3/2}\to L^{3/2}}
  \lesssim_d 3(1+b)^{1/3}2^{-bL2/3}
\]
and the $L^2$-estimate \eqref{eq:superlac-block-L2}. The identities
\[
 \frac{1-s}{3}=1-\theta_p
 \quad\text{and}\quad 
 \frac{2(1-s)}3+s=\theta_p
\]
give $\|T_{b,\varepsilon}\|_{L^p\to L^p}
\lesssim_d(1+b)^{1-\theta_p}2^{-bL\theta_p}$. The remaining numerical constant is bounded and may be absorbed into $p'$.
Since $\theta_p\geq p-1$, the interpolation arguments show that
\begin{equation}\label{eq:superlac-randomized-Lp}
  \|T_{b,\varepsilon}F\|_{L^p}
  \lesssim_d p'(1+b)e^{-c_d(p-1)Lb}\|F\|_{L^p},\quad\forall\,\,1<p\leq2.
\end{equation}

Applying Khintchine's inequality to
\eqref{eq:superlac-randomized-Lp} and bounding the pointwise
supremum by the square function gives
\begin{align}
 \left \|\sup_{j\in J}|K_{r_j,b}\bigstar F|\right\|_{L^p}
  &\leq
  \left\|\left(\sum_{j\in J}|K_{r_j,b}\bigstar F|^2\right)^{1/2}\right\|_{L^p}\lesssim_d p'(1+b)e^{-c_d(p-1)Lb}\|F\|_{L^p},
  \label{eq:superlac-block-Lp}
\end{align}
for all $1<p\leq2$.
It remains to sum over $b$. The geometric series identity
\[
  \sum_{b=0}^{\infty}(1+b)e^{-c_d(p-1)Lb}
  =\frac1{(1-e^{-c_d(p-1)L})^2}
  \lesssim_d1+\frac{1}{((p-1)L)^2}
\]
together with \eqref{eq:superlac-low-HL},
\eqref{eq:superlac-block-Lp}, and
\eqref{eq:superlac-telescope} gives \eqref{eq:superlac-main}.
To justify the passage to the infinite sum, apply Minkowski's
inequality first to finite sums in $b$. For Schwartz $F$ and finite
$J$, the telescoping partial sums converge to the spherical averages;
Fatou's lemma gives the asserted estimate. Density then extends it
to $L^p$, with the same constant independent of $J$.
\end{proof}
\smallskip

\begin{proposition}\label{boundforB}
Let $d\geq5$, $N\geq2$, and $\lambda_k=\mu_k!$, where the positive integers
$\mu_k$ are strictly increasing. Set
\begin{equation}\label{eq:LN}
  L_N=\left\lfloor\frac32\log_2N\right\rfloor.
\end{equation}
Then, for every $1<p\leq2$,
\begin{equation}\label{eq:good-tail}
  \left\|\sup_{k:\,\mu_k\geq N^3}|B_{\lambda_k,N}f|\right\|_{\ell^p}
  \lesssim_d p'
  \left(1+\frac1{((p-1)L_N)^2}\right)\|f\|_{\ell^p}.
\end{equation}
The constant is independent of $N$ and of the sequence.
\end{proposition}
\begin{proof}
The choice of $L_N$ in \eqref{eq:LN} gives
\[
 1\leq L_N\leq\tfrac32\log_2N<L_N+1\quad\text{and}\quad 
 2^{L_N}\leq N^{3/2}<2^{L_N+1}.
\]
In particular, $L_N=1$ when $N=2$. For all $\mu_k\geq N^3$, we have
\[
  \frac{\lambda_{k+1}^{1/2}}{\lambda_{k}^{1/2}}
  =\left(\frac{\mu_{k+1}!}{\mu_{k}!}\right)^{1/2}
  \geq(
  \mu_k+1)^{1/2}
  \geq(N^3+1)^{1/2}
  \geq2^{L_N}.
\]
Let $J\subset \{k>0:\,\mu_k\geq N^3\}$ be a given finite set. 
For Proposition~\ref{sampling}, take
\[
 q=Q,\qquad B_1=\mathbb C\quad\text{and}\quad B_2=\ell^\infty(J),
\]
and the finite vector multiplier
\[
 m_J(\eta)z
 :=\left\{\widetilde\psi_{2Q}(\eta)
 \widetilde{d\sigma_{\lambda_k}}(\eta)z\right\}_{k\in J}.
\]
Set
$P_QF=\mathcal F_{\R^d}^{-1}
[\widetilde\psi_{2Q}\widetilde F]$.
The continuous operator associated with $m_J$ maps $F$ to
$\{d\sigma_{\lambda_k}\bigstar P_QF\}_{k\in J}$. The multiplier is
supported in
$\{|\eta|\leq1/(4Q)\}\subset(-1/(2Q),1/(2Q)]^d$, its
$Q$-periodization has components $t_{\lambda_k,Q}$, and Young's inequality
gives $\|P_Q\|_{L^p\to L^p}\lesssim_d1$, independently of $Q$.
Therefore Theorem~\ref{thm:superlac}, Proposition~\ref{sampling}, and
\eqref{eq:factor} give
\[
 \left\|\sup_{k\in J}|B_{\lambda_k,N}f|\right\|_{\ell^p}
 \lesssim_d p'\left(1+\frac1{((p-1)L_N)^2}\right)
 \|(U_Qf)_{ext}\|_{L^p},
\]
uniformly in $J$ and $\#J$. Now \eqref{contidiscrete},
Lemma~\ref{lem:U-bound}, and the standing finite-supremum convention give
\eqref{eq:good-tail}.
\end{proof}

\subsection{The effective MSW decomposition}
We now summarize the decomposition and estimates established above.
A related decomposition into an initial segment, a main term and an
error term appears in Hughes \cite[Section~4]{Hughes17} for sparse
$k$-spherical maximal operators. For factorial sequences, see also
\cite[Section~4]{MR3982252}. Here we keep the dependence on $p$ and
on the counting function explicit.

\begin{proposition}\label{prop:klm-decomp}
Let $d\geq5$, $N\geq2$, and $\lambda_k=\mu_k!$, where the positive integers
$\mu_k$ are strictly increasing. Set $L_N=\left\lfloor\frac32\log_2N\right\rfloor$. 
There exist nonnegative-valued sublinear operators $A^{(1)}_{\star,N}$ and $A^{(2)}_{\star,N}$ such that for every finitely supported $f$,
\begin{equation}\label{eq:good-bad}
  A_\star f\leq A^{(1)}_{\star,N}f+A^{(2)}_{\star,N}f,
\end{equation}
and, for $1<p\leq2$,
\begin{align}
  \|A^{(1)}_{\star,N}f\|_{\ell^p}
  &\lesssim_d
  \left[
    \Nmu(N^3)^{1/p}
    +p'\left(1+\frac1{((p-1)L_N)^2}\right)
  \right]\|f\|_{\ell^p},\label{eq:G-bound}\\
  \|A^{(2)}_{\star,N}f\|_{\ell^2}
  &\lesssim_dN^{-(d-4)/2}\|f\|_{\ell^2}.\label{eq:H-bound}
\end{align}
\end{proposition}
\begin{proof}

By the triangle inequality and
$A_{\lambda_k}=\rho_{\lambda_k}\overline A_{\lambda_k}$, we have
\begin{align*}
 A_\star f
 &\leq \sup_{\mu_k<N^3}|A_{\lambda_k}f|
      +\sup_{\mu_k\geq N^3}|A_{\lambda_k}f|\\
 &\leq \sup_{\mu_k<N^3}|A_{\lambda_k}f|
 +R_d^{\rm norm}\left(
   \sup_{\mu_k\geq N^3}|B_{\lambda_k,N}f|
  +\sup_{\mu_k\geq N^3}
   |\overline A_{\lambda_k}f-B_{\lambda_k,N}f|
 \right).
\end{align*}
Define
\begin{align*}
 A^{(1)}_{\star,N}f
 &:=
 \sup_{\mu_k<N^3}|A_{\lambda_k}f|
 +R_d^{\rm norm}\sup_{\mu_k\geq N^3}|B_{\lambda_k,N}f|\quad\text{and}\\
 A^{(2)}_{\star,N}f
 &:=
 R_d^{\rm norm}\sup_{\mu_k\geq N^3}
 |\overline A_{\lambda_k}f-B_{\lambda_k,N}f|.
\end{align*}
A supremum over an empty index set is understood to be zero.
The $\ell^p$-norm of $\sup_{\mu_k<N^3}|A_{\lambda_k}f|$ is dominated by
$(\mathcal{N}_\mu(N^3))^{1/p}\cdot\|f\|_{\ell^p}$ for all $1<p\leq 2$.
Indeed,
\begin{align*}
    \left\|\sup_{\mu_k<N^3}|A_{\lambda_k}f|\right\|^p_{\ell^p}&:=\sum_{x\in\Z^d}\left(\sup_{\mu_k<N^3}|A_{\lambda_k}f|(x)\right)^p\\
    &\leq \sum_{x\in\Z^d}\Bigg(\sum_{k:\mu_k<N^3}|A_{\lambda_k}f(x)|^p\Bigg).
\end{align*}
Since $A_{\lambda_k}f$ is the normalized average, 
H\"older's inequality
gives
\begin{align*}
    \left\|\sup_{\mu_k<N^3}|A_{\lambda_k}f|\right\|^p_{\ell^p}
    \leq\sum_{x\in\Z^d}\Bigg(\sum_{k:\mu_k<N^3}\frac{1}{s_{\lambda_k}}\sum_{\substack{n\in\Z^d\\|n|^2=\lambda_k}} |f(x-n)|^p\Bigg)\lesssim_d\mathcal{N}_\mu(N^3)\|f\|^p_{\ell^p}.
\end{align*}

Combining this bound with Lemma~\ref{lem:tail-error} and
Proposition~\ref{boundforB} completes the proof.
\end{proof}
\medskip

\subsection{An endpoint decomposition}
For the endpoint estimate, we further decompose the continuous model
used in Proposition~\ref{prop:klm-decomp}. Using the kernels $L_r$ and
$K_{r,b}$ from the proof of Theorem~\ref{thm:superlac}, we prove a weak
type $(1,1)$ bound for the maximal operator associated with
$L_r+K_{r,0}$ and sum the remaining blocks $b\ge1$ in $L^2$.
A decomposition with weak type $(1,1)$ and $\ell^2$ bounds
also appears in Ionescu's proof of the endpoint estimate for the
full discrete spherical maximal operator \cite[Lemma~2]{Ionescu04}.

\begin{lemma}\label{lem:continuous-endpoint}
Let $d\ge5$, let $L\ge1$ be an integer, and let
$\mathcal R=\{r_j:j\in J\}$ be a finite family satisfying
\eqref{cond:superlac}. Define
\begin{align*}
 \mathcal W_{\mathcal R,L}F
   &:=\sup_{j\in J}|(L_{r_j}+K_{r_j,0})\bigstar F|,\\
 \mathcal V_{\mathcal R,L}F
   &:=\sup_{j\in J}|(d\sigma_{r_j^2}-L_{r_j}-K_{r_j,0})\bigstar F|.
\end{align*}
Then
\begin{equation}\label{eq:continuous-endpoint}
 \|\mathcal W_{\mathcal R,L}F\|_{L^{1,\infty}}
       \le C_d\|F\|_1\quad\text{and}\quad 
 \|\mathcal V_{\mathcal R,L}F\|_2
       \le C_d2^{-L}\|F\|_2.
\end{equation}
The constants are independent of $L,J$, and the radii.
\end{lemma}
\begin{proof}
The kernel bound for $L_r$ preceding \eqref{eq:superlac-low-HL}
shows that the associated maximal operator is dominated by
$C_d\MHL$, which is of weak type $(1,1)$. For $K_{r_j,0}$, use the already proved randomized estimate
\eqref{eq:superlac-block-weak11} at $b=0$.

For a finite complex sequence $(z_j)$ and independent Rademacher signs,
put $Z=\sum_j\varepsilon_jz_j$ and $s_z^2=\sum_j|z_j|^2$. Then
\[
 \mathbb E|Z|^2=s_z^2\quad\text{and}\quad  \mathbb E|Z|^4\le3s_z^4.
\]
If $s_z>0$, the Paley--Zygmund estimate gives
$\mathbb P\{|Z|>s_z/\sqrt2\}\ge1/12$. Indeed, Cauchy--Schwarz gives
\[
 s_z^2\le s_z^2/2+
       (\mathbb E|Z|^4)^{1/2}
       \mathbb P\{|Z|>s_z/\sqrt2\}^{1/2}.
\]
Apply this pointwise with $z_j=K_{r_j,0}\bigstar F(x)$. Fubini and
\eqref{eq:superlac-block-weak11} imply, for $s>0$,
\begin{align*}
 \left|\left\{x:\Big(\sum_{j\in J}
           |K_{r_j,0}\bigstar F(x)|^2\Big)^{1/2}>s\right\}\right|
 &\le12\,\mathbb E
       |\{x:|T_{0,\varepsilon}F(x)|>s/\sqrt2\}|\\
 &\le C_ds^{-1}\|F\|_1.
\end{align*}
The supremum is dominated by this square function. Combining this bound with the
low-frequency weak estimate proves the first bound in
\eqref{eq:continuous-endpoint}.

For $b\ge1$, taking expectations of the squared estimate
\eqref{eq:superlac-block-L2} gives
\[
 \left\|\Big(\sum_{j\in J}|K_{r_j,b}\bigstar F|^2\Big)^{1/2}\right\|_2
 \le C_d2^{-bL}\|F\|_2.
\]
Thus the $b\ge1$ series converges in $L^2(\R^d;\ell^\infty(J))$.
By \eqref{eq:superlac-telescope} its sum is the remainder defining
$\mathcal V_{\mathcal R,L}$, and
$\sum_{b\ge1}2^{-bL}\le2\cdot2^{-L}$ gives the second bound.
The argument starts with Schwartz functions; approximation gives the
assertions on $L^1$ and $L^2$, respectively, with the same constants.
\end{proof}

We next apply the weak and strong sampling estimates to this
decomposition on the factorial tail. The resulting $L^2$ remainder
is combined with the approximation error in Lemma~\ref{lem:tail-error}.

\begin{proposition}\label{prop:endpoint-decomp}
Let $d\ge5$ and $\lambda_k=\mu_k!$, with $(\mu_k)$ a strictly increasing
sequence of positive integers. Put
\[
 \delta_d=\min\left\{\frac{d-4}{2},\frac32\right\}.
\]
For every integer $N\ge2$, there exist nonnegative sublinear operators
$S_N,W_N,R_N$ such that, for every finitely supported $f$,
\begin{equation}\label{eq:endpoint-split}
 A_\star f\le S_Nf+W_Nf+R_Nf,
\end{equation}
and
\begin{align}
 \|S_Nf\|_{\ell^1}&\le\Nmu(N^3)\|f\|_{\ell^1},\label{eq:endpoint-S}\\
 \|W_Nf\|_{\ell^{1,\infty}}&\le C_d\|f\|_{\ell^1},\label{eq:endpoint-W}\\
 \|R_Nf\|_{\ell^2}&\le C_dN^{-\delta_d}\|f\|_{\ell^2}.
 \label{eq:endpoint-R}
\end{align}
The constants are independent of $N$ and the sequence.
\end{proposition}
\begin{proof}
Define $S_Nf=\sup_{\mu_k<N^3}A_{\lambda_k}|f|$, with an empty
supremum equal to zero. Since each average preserves mass,
\eqref{eq:endpoint-S} follows by summing its $\ell^1$ norm.

For the tail put $Q=N!$, $L=L_N$ as in \eqref{eq:LN}, and
$r_k=\sqrt{\lambda_k}$. Strict increase of the integer parameters gives
\[
 r_{k+1}/r_k\ge\sqrt{\mu_k+1}\ge N^{3/2}\ge2^{L_N}
 \quad\text{when }\mu_k\ge N^3.
\]
Fix a finite set $J$ of tail indices. Let $P_QF=\psi_{2Q}\bigstar F$
and define
\begin{align*}
 m^w_{k,N}(\xi)
  &:=\widetilde\psi_{2Q}(\xi)
       [\widetilde L_{r_k}(\xi)+\widetilde K_{r_k,0}(\xi)],\\
 m^v_{k,N}(\xi)
  &:=\widetilde\psi_{2Q}(\xi)
       [\widetilde{d\sigma_{\lambda_k}}(\xi)
              -\widetilde L_{r_k}(\xi)-\widetilde K_{r_k,0}(\xi)].
\end{align*}
These smooth multipliers are supported in $[-1/(4Q),1/(4Q)]^d$.
Since $\|\psi_{2Q}\|_1=\|\psi\|_1$, the operator $P_Q$ has uniform
$L^1$ and $L^2$ bounds. Lemma~\ref{lem:continuous-endpoint}, applied
to $P_QF$, gives continuous maximal bounds:
a weak type $(1,1)$ bound with constant $C_d$ for $m^w$ and
a strong type $(2,2)$ bound with constant $C_d2^{-L_N}$ for $m^v$.

Let $D^w_{k,N}$ and $D^v_{k,N}$ be the discrete operators obtained by
$Q$-periodizing these multipliers. Their sum is $T_{\lambda_k,Q}$.
For the weak part, apply Theorem~\ref{thm:sampling} with $p=1$; for
the remainder, apply the strong sampling estimate in
Proposition~\ref{sampling} with $p=2$
and target space $\ell^\infty(J)$. The results are
\begin{align*}
 \left\|\sup_{k\in J}|D^w_{k,N}h|\right\|_{\ell^{1,\infty}}
   \le C_d\|h\|_{\ell^1}\quad\text{and}\quad 
 \left\|\sup_{k\in J}|D^v_{k,N}h|\right\|_{\ell^2}
   \le C_d2^{-L_N}\|h\|_{\ell^2},
\end{align*}
uniformly in $N,Q,J$. Increasing $J$ gives the same bounds for the
full tail suprema.

Using the normalization factor $R_d^{\rm norm}$, set
\begin{align*}
 W_Nf&:=R_d^{\rm norm}
       \sup_{\mu_k\ge N^3}|D^w_{k,N}U_Qf|\quad\text{and}\\
 R_Nf&:=R_d^{\rm norm}\left(
       \sup_{\mu_k\ge N^3}|D^v_{k,N}U_Qf|
       +\sup_{\mu_k\ge N^3}|(\overline A_{\lambda_k}
                                      -B_{\lambda_k,N})f|\right).
\end{align*}
The factorization \eqref{eq:factor} and the identity
\[
 A_{\lambda_k}
 =\rho_{\lambda_k}\bigl(B_{\lambda_k,N}
                      +\overline A_{\lambda_k}-B_{\lambda_k,N}\bigr)
\]
prove \eqref{eq:endpoint-split}. Lemma~\ref{lem:U-bound} at $p=1,2$
gives the weak bound and, together with Lemma~\ref{lem:tail-error},
\[
 \|R_Nf\|_2\le C_d\bigl(2^{-L_N}+N^{-(d-4)/2}\bigr)\|f\|_2.
\]
Finally, $\frac32\log_2N-1<L_N\le\frac32\log_2N$ implies
$2^{-L_N}\le2N^{-3/2}$, proving \eqref{eq:endpoint-R}.
The sampling estimates extend to $U_Qf$ by approximation, as in
Section~\ref{sec:prelim}. All limiting estimates are uniform in the
finite tail set. The new exponent $\delta_d$ includes the continuous
remainder; the arithmetic error in \eqref{eq:tail-error} retains its
original exponent $(d-4)/2$.
\end{proof}
\medskip

\section{Sufficient conditions for endpoint estimates}\label{sec:endpoint}
We apply Proposition~\ref{prop:endpoint-decomp} at each amplitude
level and then sum the resulting estimates. The weak summation
inequality below allows us to sum the weak type $(1,1)$ terms
with a logarithmic loss in the level index.

\subsection{A weak type summation inequality}
\begin{lemma}\label{lem:weak-sum}
Let $u_j\ge0$ be functions on $\Z^d$ and let $w_j>0$ satisfy
$\sum_jw_j\le1$. Then
\begin{equation}\label{eq:weak-sum}
 \left\|\sum_ju_j\right\|_{\ell^{1,\infty}}
 \le4\sum_j\|u_j\|_{\ell^{1,\infty}}
                         [1+\log(1/w_j)].
\end{equation}
For $j\ge0$, the right side is at most
$C\sum_j\log(j+2)\|u_j\|_{\ell^{1,\infty}}$ with an appropriate
choice of weights.
\end{lemma}
\begin{proof}
First consider a finite sum and write
$a_j=\|u_j\|_{\ell^{1,\infty}}$. Fix $s>0$ and remove
$E_s=\bigcup_j\{u_j>s\}$, whose cardinality is at most
$s^{-1}\sum_ja_j$. On its complement, the terms with $u_j\le sw_j$
sum to at most $s$. Thus $\sum_ju_j>2s$ outside $E_s$ implies
\[
 \sum_ju_j\one_{\{sw_j<u_j\le s\}}>s.
\]
By the layer-cake formula,
\begin{align*}
 \sum_xu_j(x)\one_{\{sw_j<u_j(x)\le s\}}
 &\le sw_j\#\{u_j>sw_j\}
             +\int_{sw_j}^{s}\#\{u_j>t\}\,dt\\
 &\le a_j[1+\log(1/w_j)].
\end{align*}
Chebyshev's inequality gives
\[
 2s\#\{\textstyle\sum_ju_j>2s\}
 \le2\sum_ja_j[2+\log(1/w_j)]
 \le4\sum_ja_j[1+\log(1/w_j)].
\]
Take the supremum over $s$. For $j\ge0$, choose
$w_j=1/[(j+1)(j+2)]$, so $\sum_{j\ge0}w_j=1$ and
$1+\log(1/w_j)\le C\log(j+2)$. Countable sums follow by monotone
convergence.
\end{proof}

\subsection{Summation over amplitude levels}
We apply the weak summation inequality to the middle term of the
endpoint decomposition, while summing the initial and remainder terms
in $\ell^1$ and $\ell^2$, respectively.

\begin{lemma}\label{prop:shell}
Let $f:\Z^d\to[0,\infty)$ be finitely supported. For $j\ge0$, define
\[
 F_j=\{x\in\Z^d:2^{j-1}<f(x)\le2^j\},\qquad
 f_j=f\one_{F_j}.
\]
Then
\[
 \#\{x:A_\star f(x)>1\}
 \le C_d\sum_{j\ge0}
       [\log(j+2)+\Nmu(2^{12(j+1)})]\|f_j\|_{\ell^1}.
\]
\end{lemma}
\begin{proof}
Let $f_{\mathrm{low}}=f\one_{\{f\le1/2\}}$. The probability
normalization gives $A_\star f_{\mathrm{low}}\le1/2$.
For the remaining shells, put $m_j=\|f_j\|_1$ and choose
$N_j=2^{4(j+1)}$ so that the $\ell^2$ remainders can be summed
by Cauchy--Schwarz. Apply Proposition~\ref{prop:endpoint-decomp}
and set
\[
 S=\sum_jS_{N_j}f_j,\qquad W=\sum_jW_{N_j}f_j\quad\text{and}\quad 
 R=\sum_jR_{N_j}f_j.
\]
Then we have $A_\star f\le1/2+S+W+R$. There are only finitely many
nonzero shells, and each radius supremum has the uniform bounds
proved in Section~\ref{sec:3}.

By \eqref{eq:endpoint-S} and Chebyshev's inequality,
\begin{equation}\label{eq:shell-initial}
 \#\{S>1/6\}\le6\sum_j\Nmu(2^{12(j+1)})m_j.
\end{equation}
Lemma~\ref{lem:weak-sum} and \eqref{eq:endpoint-W} imply
\begin{equation}\label{eq:shell-weak}
 \#\{W>1/6\}\le C_d\sum_j\log(j+2)m_j.
\end{equation}
For the remainder, $f_j\le2^j$ gives $\|f_j\|_2\le2^{j/2}m_j^{1/2}$.
Since $\delta_d\ge1/2$, Minkowski and Cauchy--Schwarz yield
\begin{align}\label{eq:shell-remainder}
 \|R\|_2
 &\le C_d\sum_j2^{-4\delta_d(j+1)+j/2}m_j^{1/2}\notag\\
 &\le C_d\left(\sum_{j\ge0}2^{-8\delta_d(j+1)+j}\right)^{1/2}
             \left(\sum_jm_j\right)^{1/2}
 \le C_d\left(\sum_jm_j\right)^{1/2}.
\end{align}
Here the geometric sum is at most
$\sum_{j\ge0}2^{-3j-4}=1/14$. Consequently
$\#\{R>1/6\}\le C_d\sum_jm_j$, which is absorbed by
\eqref{eq:shell-weak} since $\log(j+2)\ge\log2$.
If $A_\star f>1$, at least one of $S,W,R$ exceeds $1/6$.
Combining these three estimates proves the lemma.
\end{proof}

\subsection{Proofs of the main results}

\begin{proposition}\label{lem:log-comparison}
If $2^{j-1}<t\le2^j$ for $j\ge0$, then
\[
 \Nmu(2^{12(j+1)})\le\Nmu(2^{24\mathcal L(t)})\quad\text{and}\quad 
 \log(j+2)\le2\log(e+\mathcal L(t)).
\]
\end{proposition}
\begin{proof}
For $j=0$, $j+1\le2\mathcal L(t)$ is immediate. For $j\ge1$,
$\log t>(j-1)\log2$ and $2\log2>1$ give
\[
 2\mathcal L(t)>2+2(j-1)\log2\ge j+1.
\]
Monotonicity of $\Nmu$ proves the first estimate. Also
$j+2\le2\mathcal L(t)+1\le(e+\mathcal L(t))^2$, which proves
the second.
\end{proof}

\subsubsection{Proof of Theorem~\ref{thm:counting}}
\begin{proof}
First suppose $f\ge0$ and $\alpha=1$. For $x\in F_j$,
Proposition~\ref{lem:log-comparison} gives
\[
 \log(j+2)+\Nmu(2^{12(j+1)})
 \le2\log(e+\mathcal L(f(x)))
                 +\Nmu(2^{24\mathcal L(f(x))}).
\]
Substituting this pointwise bound into the estimate in
Lemma~\ref{prop:shell} and adding the nonnegative contribution
from $\{f\le1/2\}$ gives
\[
 \#\{x:A_\star f(x)>1\}
 \le C_d\sum_x f(x)
       [\log(e+\mathcal L(f(x)))+\Nmu(2^{24\mathcal L(f(x))})].
\]
For general complex $f$ and $\alpha>0$, apply this to $|f|/\alpha$
and use $A_\star f\le A_\star|f|$. This proves
\eqref{eq:counting-endpoint}.
\end{proof}

\subsubsection{Proof of Theorem~\ref{thm:full-mu}}
\begin{proof}
For $u\ge1$, the counting hypothesis gives
\[
 \Nmu(2^{24u})\le C_{\log,\gamma}(1+24u\log2)^\gamma
 \le(1+24\log2)^\gamma C_{\log,\gamma}u^\gamma.
\]
Moreover,
\[
 \log(e+u)\le\log(e+1)+\log u\quad\text{and}\quad 
 \sup_{u\ge1}\frac{\log u}{u^\gamma}=\frac1{e\gamma}.
\]
Insert these bounds with $u=\mathcal L(|f(x)|/\alpha)$ in
\eqref{eq:counting-endpoint}. Since $u\ge1$, this proves
\eqref{eq:gamma-modular} with its displayed constant. The argument
holds for every fixed $\gamma>0$; the factor $\gamma^{-1}$ does not
give a uniform limit at $\gamma=0$.
\end{proof}

\subsubsection{Proof of Corollary~\ref{cor:model}}
\begin{proof}
For $\mu_k=2^k$, with $k\geq1$, one has $ \Nmu(N)
    =
    \lfloor\log_2N\rfloor$ for all $N\geq 2$.
Hence,
\[
    \frac{\Nmu(N)}{1+\log N}
    \leq
    \frac{\log N}{(1+\log N)\log2}
    \leq
    \frac1{\log2}.
\]
Since $\Clog$ is bounded by an absolute constant,
Theorem~\ref{thm:full-mu} with $\gamma=1$ proves
\eqref{eq:model-endpoint}.
\end{proof}

\begin{remark}
In the previous version of this paper
\cite[Proposition~1.3 and Section~4]{LLLS26}, we approached the
endpoint problem using a summation argument combined with a variant
of Yano's extrapolation theorem. For the sequence $\lambda_k=(2^k)!$,
this method yielded the distributional estimate
\[
 \#\{x\in\Z^d:A_\star f(x)>\alpha\}
 \lesssim_d
 \sum_{x\in\Z^d}\frac{|f(x)|}{\alpha}
 \left(1+\log^+\frac{|f(x)|}{\alpha}\right)^2,
\]
for every $\alpha>0$ and every finitely supported
$f:\Z^d\to\mathbb C$. The present improvement combines the endpoint decomposition in
Proposition~\ref{prop:endpoint-decomp} with the weak summation inequality
\eqref{eq:weak-sum}. The first continuous frequency block is treated using a weak type $(1,1)$
estimate, and the higher blocks form a remainder with a decaying
$\ell^2$ bound. This changes the amplitude loss in Lemma~\ref{prop:shell}
from $j+1$ to $\log(j+2)$ and yields the sufficient estimates
for every fixed $\gamma>0$. The quantitative strong decomposition in
Proposition~\ref{prop:klm-decomp} remains valid alongside this endpoint
refinement.
\end{remark}
\medskip

\section{Necessary conditions for endpoint estimates}\label{sec:counterexample}
\subsection{A preliminary counting estimate}
We first isolate two facts used both in Proposition A and in the refined
argument below. The quantities
$A^{\Z^d_p}_{\lambda_j}\mathbf{1}_{p\Z^d}$ are the distributions of
spherical lattice points modulo $p$; see also \cite{MR2287111}. We write
\begin{align}\label{dismodp}
 A^{\Z^d_p}_{\lambda_j}\mathbf{1}_{p\Z^d}(a)
 :=\frac1{s_{\lambda_j}}
   \#\{n\in\Z^d:|n|^2=\lambda_j,\ n\equiv a\pmod p\}
 =\nu_{\lambda_j,p}(a).
\end{align}

\begin{lemma}\label{nec:lem-periodic}
Assume \eqref{eq:necessary-endpoint} for a fixed $\gamma>0$.
For every prime $p$, every finite set of indices $I$, and every $t>0$,
\begin{align}\label{claim}
 \#\left\{a\in\Z_p^d:\max_{j\in I}\nu_{\lambda_j,p}(a)>t\right\}
 \le K\Phi_\gamma(1/t)\quad\text{and}\quad  \Phi_\gamma(u)=u\mathcal L(u)^\gamma.
\end{align}
\end{lemma}
\begin{proof}
The assertion is immediate if $I$ is empty. Otherwise fix $p,I$, put
$r=\lceil\max_{j\in I}\sqrt{\lambda_j}\rceil$, and, for integers $M$
with $Mp>2r$, set
\[
 Q_M=\{0,\ldots,Mp-1\}^d,\qquad
 f_M=\one_{p\Z^d\cap Q_M}\quad\text{and}\quad 
 Q_M^-=\{r,\ldots,Mp-1-r\}^d.
\]
For $x\in Q_M^-$ every translate $x-n$ in each of the averages belongs
to $Q_M$, so $A_{\lambda_j}f_M(x)=\nu_{\lambda_j,p}(x\bmod p)$.
Each residue class has $M^d$ representatives in $Q_M$, and
\[
 \#(Q_M\setminus Q_M^-)
 =(Mp)^d-(Mp-2r)^d\le2dr(Mp)^{d-1}.
\]
Put $V_t=\{a:\max_{j\in I}\nu_{\lambda_j,p}(a)>t\}$. The hypothesis
and $\#(p\Z^d\cap Q_M)=M^d$ therefore give
\begin{align}\label{divide}
 M^d\#V_t-2dr(Mp)^{d-1}
 &\le\#\{x\in Q_M^-:A_\star f_M(x)>t\}\notag\\
 &\le K M^d\Phi_\gamma(1/t).
\end{align}
Divide by $M^d$ and let $M\to\infty$, keeping $p,I$ fixed. The error
$2drp^{d-1}/M$ tends to zero, which proves \eqref{claim} without a
factor $p^d$.
\end{proof}

\begin{lemma}\label{nec:lem-prime}
There is an absolute integer $N_0$ such that, for every integer
$N\ge N_0$, an odd prime $p\in(N^4,2N^4)$ makes all factorials
$\mu_k!$ with $\mu_k\le N$ nonzero and pairwise distinct modulo $p$.
The threshold is independent of the sequence.
\end{lemma}
\begin{proof}
Put $R=\Nmu(N)\le N$ and
$D_N=\prod_{1\le i<j\le R}(\mu_j!-\mu_i!)$, with empty product one.
The factors are positive, since factorials of positive integers are
strictly increasing. Thus
\[
 \log D_N\le\binom R2\log(N!)\le\tfrac12N^3\log N.
\]
If $s$ distinct primes greater than $N^4$ divide $D_N$, then
$4s\log N\le\log D_N$, whence $s\le N^3/8$.
The prime number theorem gives at least $cN^4/\log N$ primes in
$(N^4,2N^4)$ for all sufficiently large $N$, with an absolute $c>0$.
This number exceeds $N^3/8$, so some such prime does not divide $D_N$.
It separates the factorials and, since $p>N$, divides none of them.
\end{proof}

For the reader's convenience, we state Proposition A again.
\begin{propA}
Let $d\geq5$ and let $\lambda_k=\mu_k!$, where $(\mu_k)$ is a strictly increasing sequence of positive integers. Suppose that, for some $C_0<\infty$,
\begin{align*}
 \#\{x\in\Z^d:A_\star f(x)>\alpha\}
 \leq C_0\sum_x\frac{|f(x)|}{\alpha}
\mathcal{L}\left(\frac{|f(x)|}{\alpha}\right)
\end{align*}
for every $\alpha>0$ and every finitely supported $f:\Z^d\to\mathbb C$. Then
\[
 \Nmu(N)\leq C_{d,C_0}(1+\log N)^2,
 \qquad N\geq2.
\]
\end{propA}
\begin{proof}
Let $N\ge N_0$ be an integer, put $R=\Nmu(N)$, and choose the prime in
Lemma~\ref{nec:lem-prime}. The case $R=0$ is immediate. The functions
$\nu_{\lambda_j,p}$, $1\le j\le R$, have pairwise disjoint supports,
since these supports lie in the distinct quadratic level sets
$\{a:|a|^2\equiv\lambda_j\pmod p\}$. Therefore,
\[
 h=\max_{1\le j\le R}\nu_{\lambda_j,p}
   =\sum_{j=1}^R\nu_{\lambda_j,p},\qquad
 0\le h\le1\quad\text{and}\quad \sum_{a\in\Z_p^d}h(a)=R.
\]
Use Lemma~\ref{nec:lem-periodic} with $\gamma=1$ and $K=C_0$.
The layer-cake formula and the trivial bound $\#\{h>t\}\le p^d$ give
\begin{align}\label{proceeding}
 R&\le1+C_0\int_{p^{-d}}^1\frac{1+\log(1/t)}{t}\,dt\notag\\
  &=1+C_0\left(d\log p+\frac{d^2}{2}(\log p)^2\right).
\end{align}
Since $p<2N^4$, this proves the asserted bound for $N\ge N_0$.
For smaller $N$, use $\Nmu(N)\le N$ and enlarge the constant.
For real $N\ge2$, apply the integer bound to $\lfloor N\rfloor$.
\end{proof}
\subsection{Asymptotics for spherical lattice points in residue classes}
To avoid the extra logarithm in \eqref{proceeding}, we seek a common
level at which every selected radius contributes many residue classes.
The asymptotic formula below, together with total mass one and a bound
for the support size, yields the estimate in
Corollary~\ref{cor:common-level}. 
\subsubsection{Spherical masses in residue classes}
For a positive integer $\lambda$ and a prime $p$, recall that
\begin{align*}
 \nu_{\lambda,p}(a)
 =\frac1{s_\lambda}
 \#\{n\in\Z^d:|n|^2=\lambda,\ n\equiv a\pmod p\},
 \qquad a\in\Z_p^d.
\end{align*}
For every positive integer $\lambda$, it satisfies
\begin{equation}\label{nec:nu-basic}
 0\leq\nu_{\lambda,p}\leq1,\quad
 \sum_{a\in\Z_p^d}\nu_{\lambda,p}(a)=1\quad\text{and}\quad
 \supp\nu_{\lambda,p}\subseteq
 \{a\in\Z_p^d:|a|^2\equiv\lambda\pmod p\}.
\end{equation}
The Fourier transform of $\nu_{\lambda,p}$ is given by
\begin{align}\label{asy1}
\widehat\nu_{\lambda,p}(b)
=\sum_{a\in\Z_p^d}\nu_{\lambda,p}(a)e_p(-a\cdot b)
 =\rho_\lambda a_\lambda(b/p),
\end{align}
where $a_\lambda$ is the multiplier of $\overline A_\lambda$ and $
 \rho_\lambda={\lambda^{d/2-1}}/{s_\lambda}$.

For $b\in\Z_p^d$ and $d\geq5$, define the approximation multiplier
\begin{equation}\label{nec:H-definition}
 h_{\lambda,p}(b)
 =\sum_{\substack{q\ge1\\qb/p\in\Z^d}}
 \sum_{a\in\Z_q^\times}e_q(-\lambda a)G(a,qb/p,q).
\end{equation}
The Gauss bound makes this series absolutely convergent, uniformly
in $b,p,\lambda$, since $\sum_{q\ge1}q^{1-d/2}<\infty$ for $d\ge5$.
Its value is independent of the representative of $b$ modulo $p$.
We next approximate $a_\lambda(b/p)$ by $\cMSW h_{\lambda,p}(b)$.

\begin{lemma}\label{lem-rational-error}
Let $d\ge5$, let $\lambda$ be a positive integer, and let $p$ be a
prime. If $\lambda\ge p^8$, then
\begin{align*}
 \sup_{b\in\Z_p^d}
 |a_\lambda(b/p)-\cMSW h_{\lambda,p}(b)|
 \leq C_d\lambda^{-1/8}.
\end{align*}
\end{lemma}
\begin{proof}
The multiplier $a_\lambda$ is a trigonometric polynomial. For each
$q$, the cutoff translates in \eqref{majorarc2} have bounded overlap,
so the Gauss bound gives
\[
 \left\|\sum_{a\in\Z_q^\times}e_q(-\lambda a)
                   m_\lambda^{a/q}\right\|_\infty
 \le C_dq^{1-d/2}.
\]
Each term is continuous, and the series defining $m_\lambda$ converges
uniformly for $d\ge5$. Hence $a_\lambda-m_\lambda$ is continuous.
Plancherel and Proposition~\ref{minor} bound its essential supremum by
$C_d\lambda^{-(d-4)/4}$; continuity gives the bound at every point,
including $b/p$. Thus,
\begin{align*}
    \left|a_\lambda(b/p)-\cMSW h_{\lambda,p}(b)\right|&\leq \left|a_\lambda(b/p)-m_\lambda(b/p)\right|+\left|m_\lambda(b/p)-\cMSW h_{\lambda,p}(b)\right|\\
    &\leq C_d\lambda^{-(d-4)/4}+\left|m_\lambda(b/p)-\cMSW h_{\lambda,p}(b)\right|.
\end{align*}
Since $d\geq 5$, it remains to estimate $\left|m_\lambda(b/p)-\cMSW h_{\lambda,p}(b)\right|$.

We apply the high-low decomposition in (\ref{low}) and (\ref{high})
to $m_\lambda$ and the analogous decomposition to $h_{\lambda,p}$.
Choose $T=\lfloor\lambda^{1/4}\rfloor$. Since $\lambda\ge p^8\ge2^8$,
\[
 \tfrac12\lambda^{1/4}\le T\le\lambda^{1/4}.
\]
Then
\begin{align*}
    \left|m_\lambda(b/p)-\cMSW h_{\lambda,p}(b)\right|\leq \left|m_{\lambda,> T}(b/p)\right|+\left|\cMSW h^{> T}_{\lambda,p}(b)\right|+\left|m_{\lambda,\leq T}(b/p)-\cMSW h^{\leq T}_{\lambda,p}(b)\right|,
\end{align*}
where 
\begin{align*}
    h^{> T}_{\lambda,p}(b)
 =\sum_{\substack{q>T\\qb/p\in\Z^d}}
 \sum_{a\in\Z_q^\times}e_q(-a\lambda)G(a,qb/p,q)\quad\text{and}\quad  h^{\leq T}_{\lambda,p}(b)= h_{\lambda,p}(b)- h^{> T}_{\lambda,p}(b).
\end{align*}
The Gauss bound $|G(a,\ell,q)|\lesssim_d q^{-d/2}$ for $(a,q)=1$
gives the following estimate for the high parts of $m_\lambda$ and $h_{\lambda,p}$:
\begin{align*}
\left|m_{\lambda,> T}(b/p)\right|+\left|\cMSW h^{> T}_{\lambda,p}(b)\right|\leq C_d\sum_{q>T}q^{1-d/2}\le C_dT^{-(d-4)/2}.
\end{align*}

Besides, from the definition of the low parts, $\left|m_{\lambda,\leq T}(b/p)-\cMSW h^{\leq T}_{\lambda,p}(b)\right|$ is bounded by
\begin{align}\label{hard}
 &\cMSW\left|\sum_{\substack{q\leq T\\qb/p\notin\Z^d}}
     \sum_{a \in \Z_q^\times} e_q(-\lambda a)
                         m^{a/q}_\lambda(b/p)\right|\notag\\
 &\quad+\cMSW\left|\sum_{\substack{q\leq T\\qb/p\in\Z^d}}
     \sum_{a \in \Z_q^\times} e_q(-\lambda a)
       \left[m^{a/q}_\lambda(b/p)-G(a,qb/p,q)\right]\right|.
\end{align}
For $q\le T$, if $qb/p\notin\Z^d$, then $b/p-\ell/q\ne0$. At least one coordinate
of this vector is a nonzero integer divided by $pq$. Therefore
\begin{equation}\label{nec:separation}
 \left|\frac bp-\frac\ell q\right|\ge\frac1{pq}.
\end{equation}
As a consequence, the Fourier decay of the spherical measure and the Gauss bound
show that the first term in (\ref{hard}) is bounded by
\begin{align*}
 C_d\sum_{q\le T}q^{1-d/2}
 \left(1+\frac{\sqrt\lambda}{pq}\right)^{-(d-1)/2}\le C_d\left(\frac{pT}{\sqrt\lambda}\right)^{(d-1)/2}
 \sum_{q\le T}q^{1-d/2}\le C_d\left(\frac{pT}{\sqrt\lambda}\right)^{(d-1)/2}.
\end{align*}
For $q\le T$, if $qb/p\in\Z^d$, the summand with $\ell=qb/p$ in
\eqref{majorarc2} has cutoff and surface Fourier transform both
equal to one. It contributes  $G(a,qb/p,q)$. Every other
contributing term satisfies $b/p-\ell/q\ne0$, and hence (\ref{nec:separation}). Thus, by the Fourier decay of the spherical measure and the Gauss bound,
the second term in (\ref{hard}) is also bounded by $C_d\left(\frac{pT}{\sqrt\lambda}\right)^{(d-1)/2}$.

Combining all the estimates and the decomposition, we obtain
\begin{align*}
  \left|a_\lambda(b/p)-\cMSW h_{\lambda,p}(b)\right|
 &\le C_d\left[
 \lambda^{-(d-4)/4}+T^{-(d-4)/2}
 +\left(\frac{pT}{\sqrt\lambda}\right)^{(d-1)/2}\right]\\
 &\le C_d\left[
 \lambda^{-(d-4)/4}+\lambda^{-(d-4)/8}
 +\lambda^{-(d-1)/16}\right]
 \le C_d\lambda^{-1/8}.
\end{align*}
Here we used both bounds for $T$, $p\le\lambda^{1/8}$, and $d\ge5$.
\end{proof}

\subsubsection{Ramanujan sums and prime-power cancellation}
We investigate the kernel $H_{\lambda,p}$ associated with $h_{\lambda,p}$, which is given by the inverse Fourier transform
\begin{equation}\label{nec:K-definition}
 H_{\lambda,p}(c)
 =p^{-d}\sum_{b\in\Z_p^d}h_{\lambda,p}(b)e_p(c\cdot b).
\end{equation}
For $q\ge1$, let $\mathfrak{R}_q(t)=\sum_{a\in\Z_q^\times}e_q(at)$ be the Ramanujan sum 
and define
\begin{equation}\label{nec:B-definition}
 \mathfrak{QR}_q(\lambda)
 =q^{-d}\sum_{n\in\Z_q^d}\mathfrak{R}_q(|n|^2-\lambda)
 =\sum_{a\in\Z_q^\times}e_q(-a\lambda)G(a,0,q).
\end{equation}
The Gauss bound gives
\begin{equation}\label{nec:B-bound}
 |\mathfrak{QR}_q(\lambda)|\le C_dq^{1-d/2}
 \quad\text{and}\quad \sum_{q\ge1}|\mathfrak{QR}_q(\lambda)|\le C_d.
\end{equation}

We reformulate $H_{\lambda,p}$ in terms of Ramanujan sums. The
$q$-th contribution to \eqref{nec:K-definition} is
\[
 C_q(c)=p^{-d}\sum_{\substack{b\in\Z_p^d\\qb/p\in\Z^d}}
 e_p(c\cdot b)\sum_{a\in\Z_q^\times}
 e_q(-a\lambda)G(a,qb/p,q).
\]
The Gauss bound gives $|C_q(c)|\le C_dq^{1-d/2}$, uniformly in
$p,\lambda,c$. Thus $\sum_q|C_q(c)|<\infty$. We may exchange the
finite $b$-sum with the denominator series and expand the finite
Gauss sums within each denominator. By \eqref{normalizedGauss},
\begin{align*}
 H_{\lambda,p}(c)= &\sum_{q\ge1}q^{-d}
 \sum_{a\in\Z_q^\times}\sum_{n\in\Z^d_q}  e_q(a(|n|^2-\lambda))\cdot\left[p^{-d}\sum_{\substack{b\in\Z_p^d\\qb/p\in\Z^d}}e_p((c+n)b)  \right]\\
 &=\sum_{q\ge1}q^{-d}
 \sum_{n\in\Z^d_q} \mathfrak{R}_q(|n|^2-\lambda) \cdot\left[p^{-d}\sum_{\substack{b\in\Z_p^d\\qb/p\in\Z^d}}e_p((c+n)b)  \right].
\end{align*}
Since $p$ is prime, either $(p,q)=1$ or $p\mid q$. In the first case,
the condition $qb/p\in\Z^d$ forces $b=0$ in $\Z_p^d$.
In the second case, the condition holds for every $b\in\Z_p^d$, and
character orthogonality gives
\[
 p^{-d}\sum_{b\in\Z_p^d}e_p((c+n)b)
 =\one_{\{(c+n)\equiv 0\pmod p\}},
\]
and hence we can rewrite $ H_{\lambda,p}(c)$ as
\begin{align*}
    p^{-d}\sum_{\substack{q\ge1\\(p,q)=1}} \mathfrak{QR}_q(\lambda)+ \sum_{\substack{q\ge1\\p\mid q}}q^{-d}
 \sum_{\substack{n\in\Z^d_q\\ (c+n)\equiv 0\pmod p}} \mathfrak{R}_q(|n|^2-\lambda) .
\end{align*}

For each $q$ with $p\mid q$, there exist $k\in\N$ and $r\in\N$ such that $q=p^kr$, where $(p,r)=1$. The Chinese remainder theorem $\Z^d_q\cong\Z^d_{p^k}\times\Z^d_r$
and the multiplicativity of Ramanujan sums,
$\mathfrak{R}_{p^kr}(t)=\mathfrak{R}_{p^k}(t)\mathfrak{R}_r(t)$, give
\begin{equation}\label{nec:CRT-factor}
q^{-d}
 \sum_{\substack{n\in\Z^d_q\\ (c+n)\equiv 0\pmod p}} \mathfrak{R}_q(|n|^2-\lambda)=\mathfrak{D}_{p^k}(c;\lambda)\cdot\mathfrak{QR}_r(\lambda),
\end{equation}
in which
\begin{equation}\label{nec:D-definition}
 \mathfrak{D}_{p^k}(c;\lambda)
 =p^{-kd}\sum_{\substack{n\in\Z_{p^k}^d\\(n+c)\equiv 0\pmod p}}
 \mathfrak{R}_{p^k}(|n|^2-\lambda).
\end{equation}
For $p\mid q$, the left side of \eqref{nec:CRT-factor} is
$C_q(c)$. Hence
\[
 \sum_{k\ge1}\sum_{\substack{r\ge1\\(p,r)=1}}
 |\mathfrak D_{p^k}(c;\lambda)\mathfrak{QR}_r(\lambda)|
 =\sum_{\substack{q\ge1\\p\mid q}}|C_q(c)|<\infty.
\]
The regrouping $q=p^kr$ is therefore justified, and
\begin{align}\label{reducedH{lambda,p}}
 H_{\lambda,p}(c)&=p^{-d}\sum_{\substack{q\ge1\\(p,q)=1}} \mathfrak{QR}_q(\lambda)+ \sum_{k\geq 1}\sum_{\substack{r\geq 1\\(p,r)=1}}\mathfrak{D}_{p^k}(c;\lambda)\cdot\mathfrak{QR}_r(\lambda).
\end{align}

The next lemma shows that, when $p$ is odd and $p\nmid\lambda$,
all terms with $k\ge2$ in the second sum in
\eqref{reducedH{lambda,p}} vanish, leaving only its $k=1$ term.
\begin{lemma}\label{lem:Raman-cancellation}
Let $\lambda$ be a positive integer and let $p$ be an odd prime
with $p\nmid\lambda$. For every $c\in\Z_p^d$,
\begin{align}
\mathfrak{D}_{p}(c;\lambda)&=p^{-d}\cdot\mathfrak{R}_p(|c|^2-\lambda),
 \label{nec:first-local}\\
\mathfrak{D}_{p^k}(c;\lambda)&=0,\quad \forall k\geq2.
 \label{nec:higher-local}
\end{align}
\end{lemma}
\begin{proof}
The case $k=1$ follows from the definition. It remains to consider the case $k\geq2$.
From the definition of the Ramanujan sum and the orthogonality property of characters, for $k\geq2$,
\begin{align*}
 \mathfrak{R}_{p^k}(t)=p^k\one_{\{p^k\mid t\}}
 -p^{k-1}\one_{\{p^{k-1}\mid t\}},
\end{align*}
and hence $p^{kd}\mathfrak{D}_{p^k}(c;\lambda)$ can be written as
\begin{align}\label{nec:cancellation-count}
 &p^k\#\{n\in\Z_{p^k}^d:n\equiv -c\pmod p,
 \ |n|^2\equiv\lambda\pmod{p^k}\}\notag\\
 &-p^{k-1}\#\{n\in\Z_{p^k}^d:n\equiv -c\pmod p,
 \ |n|^2\equiv\lambda\pmod{p^{k-1}}\}\notag\\
 &=p^k\#\{n\in\Z_{p^k}^d:n\equiv -c\pmod p,
 \ |n|^2\equiv\lambda\pmod{p^k}\}\notag\\
 &-p^{k-1+d}\#\{m\in\Z_{p^{k-1}}^d:m\equiv -c\pmod p,
 \ |m|^2\equiv\lambda\pmod{p^{k-1}}\}.
\end{align}
If $|c|^2\not\equiv\lambda\pmod p$, both counts vanish. Suppose
$|c|^2\equiv\lambda\pmod p$. A solution $m$ modulo $p^{k-1}$
has lifts $n=m+p^{k-1}z$, $z\in\Z_p^d$. Since $k\ge2$,
\[
 |n|^2\equiv |m|^2+2p^{k-1}m\cdot z\pmod{p^k},
\]
and the lift condition is
\[
 2m\cdot z\equiv\frac{\lambda-|m|^2}{p^{k-1}}\pmod p.
\]
Because $p$ is odd and $|m|^2\equiv\lambda\not\equiv0\pmod p$,
at least one coefficient $2m_i$ is nonzero modulo $p$. Choosing the
other $d-1$ coordinates freely determines $z_i$ uniquely. Thus each
solution has $p^{d-1}$ lifts, and

\begin{align*}
  &\#\{n\in\Z_{p^k}^d:n\equiv -c\pmod p,
 \ |n|^2\equiv\lambda\pmod{p^k}\}\\&=p^{d-1}\#\{m\in\Z_{p^{k-1}}^d:m\equiv -c\pmod p,
 \ |m|^2\equiv\lambda\pmod{p^{k-1}}\}.
\end{align*}
Substituting into \eqref{nec:cancellation-count} gives $p^{kd}\mathfrak{D}_{p^k}(c;\lambda)=0$ for all $k\geq2$ and the proof is complete.
\end{proof}

\begin{proposition}\label{prop-exact-main}
Let $d\ge5$ and let $\lambda$ be a positive integer. If $p$ is an odd prime and
$p\nmid\lambda$, then
\begin{equation}\label{nec:exact-main}
 H_{\lambda,p}(c)
 =p^{1-d}\one_{\{|c|^2\equiv\lambda\pmod p\}}
 \sum_{\substack{r\geq 1\\(p,r)=1}}\mathfrak{QR}_r(\lambda).
\end{equation}
In particular, $|H_{\lambda,p}(c)|\le C_dp^{1-d}$ uniformly in $c$.
\end{proposition}
\begin{proof}
 It follows from \eqref{reducedH{lambda,p}} and
Lemma~\ref{lem:Raman-cancellation} that
\begin{align*}
H_{\lambda,p}(c)&=p^{-d}\sum_{\substack{q\ge1\\(p,q)=1}} \mathfrak{QR}_q(\lambda)+ \mathfrak{D}_{p}(c;\lambda)\cdot\sum_{\substack{r\geq 1\\(p,r)=1}}\mathfrak{QR}_r(\lambda)\\
&=p^{-d}\left[1+\mathfrak{R}_p(|c|^2-\lambda)\right]\sum_{\substack{r\geq 1\\(p,r)=1}}\mathfrak{QR}_r(\lambda).
\end{align*}
For a prime $p$, one has $\mathfrak{R}_p(t)=p-1$ if $p\mid t$, and $\mathfrak{R}_p(t)=-1$
otherwise. Thus, $1+\mathfrak{R}_p(t)=p\one_{\{p\mid t\}}$, which proves
\eqref{nec:exact-main}. The remaining estimate follows from \eqref{nec:B-bound}.
\end{proof}

\subsubsection{Asymptotic behavior for $\nu_{\lambda,p}$}
Combining (\ref{asy1}), Lemma~\ref{lem-rational-error}, and
Proposition~\ref{prop-exact-main}, we obtain the asymptotic formula
stated in Theorem~\ref{thm:asym}. We repeat it here for convenience.
\begin{theorem}
Let $d\ge5$, let $\lambda$ be a positive integer, and let $p$ be an
odd prime with $p\nmid\lambda$. If $\lambda\ge p^8$,
then, uniformly for $a\in\Z_p^d$,
\begin{align*}
 \nu_{\lambda,p}(a)
 =\rho_\lambda\cMSW p^{1-d}\one_{\{|a|^2\equiv\lambda\pmod p\}}
 \sum_{\substack{r\geq 1\\(p,r)=1}}\mathfrak{QR}_r(\lambda)+O_d(\lambda^{-1/8}).
\end{align*}
\end{theorem}

\begin{proof}
Fourier inversion in \eqref{asy1} and Lemma~\ref{lem-rational-error} give
\[
 \left|\nu_{\lambda,p}(a)-\rho_\lambda\cMSW H_{\lambda,p}(a)\right|
 \le \rho_\lambda p^{-d}\sum_{b\in\Z_p^d}C_d\lambda^{-1/8}
 \le C_d\lambda^{-1/8}.
\]
Here the factor $p^{-d}$ cancels the number $p^d$ of terms in the sum,
and $\rho_\lambda\le R_d^{\rm norm}$.
Apply Proposition~\ref{prop-exact-main} to obtain
\eqref{nec:modular-asymptotic}.
\end{proof}

\begin{corollary}\label{cor:common-level}
Let $d\ge5$, let $\lambda$ be a positive integer, and let $p$ be an odd prime.
If $p\nmid\lambda$ and $\lambda\ge p^{8d}$, there are constants
$B_d\ge1$ and $c_d=(2B_d)^{-1}>0$, depending only on $d$, such that
$\|\nu_{\lambda,p}\|_\infty\le B_dp^{1-d}$ and
\begin{equation}\label{nec:level-cardinality}
 \#\{a\in\Z_p^d:\nu_{\lambda,p}(a)>1/{4p^{d-1}}\}
\ge c_dp^{d-1}.
\end{equation}
\end{corollary}
\begin{proof}
Recall from (\ref{nec:nu-basic}) that $\{a\in\Z_p^d:|a|^2\equiv\lambda\pmod p\}$ contains the support of $\nu_{\lambda,p}$.
Fixing the first $d-1$ coordinates leaves a quadratic equation in the last coordinate,
with at most two roots. Hence, this set has at most $2p^{d-1}$ elements, and
\[
 \sum_{\{a:\nu_{\lambda,p}(a)\le 1/{4p^{d-1}}\}}\nu_{\lambda,p}(a)
 \leq 1/{4p^{d-1}}\cdot\#\{a\in\Z_p^d:|a|^2\equiv\lambda\pmod p\}\leq\frac12,
\]
which, together with the total mass identity in (\ref{nec:nu-basic}), gives
$$ \frac12
 \leq1-\sum_{\{a:\nu_{\lambda,p}(a)\le 1/{4p^{d-1}}\}}\nu_{\lambda,p}(a)=\sum_{\{a:\nu_{\lambda,p}(a)>1/{4p^{d-1}}\}}\nu_{\lambda,p}(a).$$

Moreover, $\lambda^{-1/8}\le p^{-d}\le p^{1-d}$.
Theorem~\ref{thm:asym}, \eqref{nec:B-bound}, and
$\rho_\lambda\le R_d^{\rm norm}$ therefore give
$\|\nu_{\lambda,p}\|_\infty\le B_dp^{1-d}$ for a fixed $B_d\ge1$.
Consequently
\[
 \frac12\le B_dp^{1-d}
       \#\{a:\nu_{\lambda,p}(a)>1/(4p^{d-1})\},
\]
which proves \eqref{nec:level-cardinality} with $c_d=(2B_d)^{-1}$.
\end{proof}

\subsection{A refined inverse argument: proof of Theorem \ref{thm:klm-counterexample}}

\begin{proposition}\label{nec:prop-finite-family}
Assume \eqref{eq:necessary-endpoint} for a fixed $\gamma>0$.
Let $p$ be an odd prime and $I$ a finite set of indices. Suppose that
$p\nmid\lambda_j$, $\lambda_j\ge p^{8d}$ for every $j\in I$, and that
the residues $\lambda_j\bmod p$ are pairwise distinct. Then
\begin{equation}\label{nec:finite-family-bound}
 \#I\le\frac4{c_d}K
             [1+\log4+(d-1)\log p]^\gamma,
\end{equation}
where $c_d$ is the constant in Corollary~\ref{cor:common-level}.
\end{proposition}
\begin{proof}
The assertion is immediate for empty $I$. Otherwise the supports of
$\nu_{\lambda_j,p}$ lie in disjoint quadratic level sets. Their
superlevel sets at $t=1/(4p^{d-1})$ are therefore disjoint, and
Corollary~\ref{cor:common-level} gives
\[
 \#\{a:\max_{j\in I}\nu_{\lambda_j,p}(a)>t\}
 \ge c_d(\#I)p^{d-1}.
\]
Lemma~\ref{nec:lem-periodic}, with this value of $t$, bounds the left
side by $4Kp^{d-1}[1+\log4+(d-1)\log p]^\gamma$.
Canceling $p^{d-1}$ proves \eqref{nec:finite-family-bound}.
\end{proof}

\begin{proof}[Proof of Theorem~\ref{thm:klm-counterexample}]
First let $N$ be an integer and set
\[
 T(N)=\left\lceil64d(1+\log N)\right\rceil\quad\text{and}\quad 
 I_N=\{k:T(N)<\mu_k\le N\}.
\]
For sufficiently large $N$, choose the prime in
Lemma~\ref{nec:lem-prime}. Once $T(N)<N$, the discarded parameters
are counted:
\begin{equation}\label{except}
 \Nmu(N)-\#I_N=\Nmu(T(N)).
\end{equation}
For every positive integer $m$, $m!\ge2^{m-1}$. If $k\in I_N$,
the integrality of $\mu_k$ and $T(N)$ gives $\mu_k-1\ge T(N)$, whence
\[
 \log(\mu_k!)\ge64d(1+\log N)\log2
 \ge8d(\log2+4\log N)>8d\log p.
\]
Thus $\lambda_k\ge p^{8d}$ for every retained index. Also, for $d\ge5$,
\[
 1+\log4+(d-1)\log p\le4(d-1)(1+\log N).
\]
Proposition~\ref{nec:prop-finite-family} and \eqref{except} give
\begin{align}\label{nec:counting-recursion}
 \Nmu(N)&\le\Nmu(T(N))
       +D_{d,\gamma}K(1+\log N)^\gamma\quad\text{and}\notag\\
 D_{d,\gamma}&=\frac4{c_d}[4(d-1)]^\gamma.
\end{align}
We retain the discarded count $\Nmu(T(N))$ in the recursion.
The bound $\Nmu(T(N))\le T(N)=O_d(1+\log N)$ is insufficient to obtain
the desired $(1+\log N)^\gamma$ bound when $0<\gamma<1$.

We now close this recursion. Write $u=1+\log N$. Since
$T(N)\le(64d+1)u$, choose an integer $N_*=N_*(d)\ge N_0$ so that
\[
 u\ge\max\{16,\,4[1+\log(64d+1)]\}\qquad(N\ge N_*).
\]
For $u\ge16$, $\log u\le u/4$, as follows by differentiation of
$u/4-\log u$. Hence
\begin{equation}\label{nec:counting-contraction}
 1+\log T(N)\le1+\log(64d+1)+\log u\le u/2,
 \qquad T(N)<N.
\end{equation}
The strict inequality follows from $u/2<u=1+\log N$.
Put
\[
 A=N_*+\frac{D_{d,\gamma}K}{1-2^{-\gamma}}.
\]
We prove $\Nmu(N)\le A(1+\log N)^\gamma$ by strong induction over
integers $N\ge2$. For $N<N_*$ use $\Nmu(N)\le N\le N_*\le A$.
For $N\ge N_*$, the induction hypothesis at the integer $T(N)<N$
and \eqref{nec:counting-recursion}--\eqref{nec:counting-contraction}
give
\begin{align*}
 \Nmu(N)
 &\le A(1+\log T(N))^\gamma
                +D_{d,\gamma}K(1+\log N)^\gamma\\
 &\le[2^{-\gamma}A+D_{d,\gamma}K](1+\log N)^\gamma\\
 &\le A(1+\log N)^\gamma.
\end{align*}
For real $N\ge2$, apply this to $\lfloor N\rfloor$, using
$\Nmu(N)=\Nmu(\lfloor N\rfloor)$ and
$1+\log\lfloor N\rfloor\le1+\log N$.
This proves the theorem, with the dimension constant $4/c_d$ in its
displayed bound. At $\gamma=1$ the factor $(1-2^{-\gamma})^{-1}$ is
$2$, recovering $C_{\log}\le C_d(1+K)$.
\end{proof}

\subsection{Distinction from weak type \texorpdfstring{$(1,1)$}{(1,1)}}
The endpoint estimates for $\gamma>0$ do not imply weak type
$(1,1)$. The same periodic test gives the following necessary condition.

\begin{proposition}\label{nec:weak-one}
Let $d\ge5$ and $\lambda_k=\mu_k!$, with $(\mu_k)$ strictly increasing
in the positive integers. Suppose that, for every finitely supported $f$
and every $\alpha>0$,
\begin{equation}\label{nec:weak-one-bound}
 \#\{x:A_\star f(x)>\alpha\}\le K\alpha^{-1}\|f\|_{\ell^1}.
\end{equation}
With $T(N)=\lceil64d(1+\log N)\rceil$, one then has
\begin{equation}\label{nec:weak-local-count}
 \Nmu(N)-\Nmu(T(N))\le\frac4{c_d}K
\end{equation}
for every sufficiently large integer $N$, with a threshold depending
only on $d$.
\end{proposition}
\begin{proof}
Apply the box argument in Lemma~\ref{nec:lem-periodic} directly to
\eqref{nec:weak-one-bound}. The same boundary error tends to zero,
and the quotient bound is now
$\#\{a:\max_{j\in I}\nu_{\lambda_j,p}(a)>t\}\le K/t$.
For a family satisfying the hypotheses of
Proposition~\ref{nec:prop-finite-family}, the common level
$t=1/(4p^{d-1})$ gives $c_d(\#I)p^{d-1}\le4Kp^{d-1}$.
Choose the separating prime and the set $I_N$ as in the proof
of Theorem~\ref{thm:klm-counterexample}. Its cardinality is
$\Nmu(N)-\Nmu(T(N))$, proving \eqref{nec:weak-local-count}.
\end{proof}

\begin{corollary}\label{nec:all-positive-not-weak}
Let $\mu_k=2^{2^k}$, $\lambda_k=\mu_k!$, $k\ge1$, and $d\ge5$.
Then \eqref{eq:necessary-endpoint} holds with a finite constant for
every fixed $\gamma>0$, but $A_\star$ is not of weak type $(1,1)$.
\end{corollary}
\begin{proof}
For $N\ge4$,
$\Nmu(N)=\lfloor\log_2\log_2N\rfloor$, while $\Nmu(N)=0$ for
$2\le N<4$. If $\mu_k\le N<\mu_{k+1}$, then
\[
 \frac{\Nmu(N)}{(1+\log N)^\gamma}
 \le\frac{k}{[(\log2)2^k]^\gamma}
 \le\frac1{e\gamma(\log2)^{\gamma+1}}.
\]
The last inequality uses
$\sup_{v\ge0}v e^{-\gamma(\log2)v}=1/(e\gamma\log2)$.
Thus $C_{\log,\gamma}<\infty$ for each $\gamma>0$, and
Theorem~\ref{thm:full-mu} gives the claimed estimates.

On the other hand, at $N=\mu_k$, we have $\Nmu(N)=k$,
\[
 T(\mu_k)\le E_d2^k\quad\text{and}\quad E_d=(64d+1)(1+\log2).
\]
The counting formula therefore gives
\[
 \Nmu(T(\mu_k))
 \le\log_2(k+\log_2 E_d).
\]
Consequently
$\Nmu(\mu_k)-\Nmu(T(\mu_k))
\ge k-\log_2(k+\log_2 E_d)\to\infty$,
contradicting Proposition~\ref{nec:weak-one} for any finite $K$.
\end{proof}

\begin{remark}
The weak $(1,1)$ test above uses the linear function $\Phi(t)=t$
directly. It is not obtained by setting $\gamma=0$ in our
results for $\gamma>0$: the sufficient constant contains $\gamma^{-1}$,
and the recursive necessity bound contains $(1-2^{-\gamma})^{-1}$.
We assert only the local necessary condition
\eqref{nec:weak-local-count}, not a general weak $(1,1)$
characterization.
\end{remark}
\medskip

\section{Restricted weak type estimates for factorial blocks}\label{sec:6}
\subsection{Proof of the sufficiency part of Theorem~\ref{thm:restricted}}
We now consider the factorial blocks \eqref{eq:block-family}.
To prove sufficiency, we combine the following approximation and
counting estimates with a restricted weak bound for the maximal
operator associated with $B_{\lambda,N}$. We obtain the latter bound
from the continuous covering theorem by weak sampling.
The estimates below account for all radii within each block.

\begin{lemma}\label{lem:block-tail}
Let $d\ge5$ and let $N\ge2$ be an integer. For the set
\eqref{eq:block-family}, define
\[
 \mathcal T_N=\bigcup_{j:n_j\ge N^3}\{(R_jt)^2:t\in F_j\}\quad\text{and}\quad K(N)=\#(\Lambda(\mathbf F)\setminus\mathcal T_N).
\]
Then, uniformly over all choices of $\mathbf F$,
\begin{equation}\label{eq:block-tail}
 \left\|\sup_{\lambda\in\mathcal T_N}
   |(\overline A_\lambda-B_{\lambda,N})f|\right\|_{\ell^2}
 \le C_dN^{-(d-4)/2}\|f\|_{\ell^2},
\end{equation}
and, for every $\eps>0$,
\begin{equation}\label{eq:initial-block-count}
 K(N)\le2^{1+\sqrt{3\log_2N}}\le C_\eps N^\eps.
\end{equation}
\end{lemma}
\begin{proof}
Put $Q=N!$ and fix a finite set $J\subset\mathcal T_N$. If
$t=m2^{-j}\in F_j$, then $\sqrt\lambda=R_jt=mn_j!$ for an integer
$m$. Hence $\sqrt\lambda\ge(N^3)!$ on the tail, and $q\mid\lambda$
for every $q\le N$. The small-denominator phases therefore equal one.
The dyadic estimate for $E_\lambda$ in Proposition~\ref{minor}, the
single-denominator bound \eqref{eq:single-denom-l2}, and its thinner-cutoff
version used in \eqref{eq:B-large-q-detail} give the first three terms
of the proof of Lemma~\ref{lem:tail-error}, each bounded by
$C_dN^{-(d-4)/2}\|f\|_2$. These estimates do not require separation
of the radii.

For the remaining cutoff difference we count all radii in each
block. Define
\[
 \Sigma_N=\sum_{\lambda\in\mathcal T_N}
               (\sqrt\lambda/Q)^{-(d-1)}.
\]
Since $\#F_j\le2^j\le n_j$ and the integers $n_j$ are distinct,
\begin{align}\label{rw:block-sum}
 \Sigma_N
 \le Q^{d-1}\sum_{j:n_j\ge N^3}
              \frac{\#F_j}{(n_j!)^{d-1}}
 &\le (N!)^{d-1}\sum_{n\ge N^3}\frac{n}{(n!)^{d-1}}\notag\\
 &\le 2N^3\left(\frac{N!}{(N^3)!}\right)^{d-1}
 \le 2N^{-(d-4)}.
\end{align}
Indeed, consecutive terms $a_n=n/(n!)^{d-1}$ satisfy
$a_{n+1}/a_n=1/[n(n+1)^{d-2}]\le1/2$. Also
$(N^3)!/N!\ge N^{N^3-N}$ and $N^3-N\ge1$, so
$N^3(N!/(N^3)!)^{d-1}\le N^{3-(d-1)}$.

For $q\le N$ and $a\in\Z_q^\times$, the cutoff difference in
\eqref{eq:cutoff-diff-mult} is supported where
$1/(8Q)\le|\xi-\ell/q|\le1/(2q)$. The stationary phase method, as used in
\eqref{eq:stationary-phase-use}, together with bounded overlap of
these translated supports, gives
\[
 \sup_{\xi\in\T^d}\sum_{\lambda\in J}
 |m^{a/q}_\lambda(\xi)-b^{a/q}_{\lambda,N}(\xi)|^2
 \le C_dq^{-d}\Sigma_N.
\]
The supremum is dominated by the square function, so Plancherel yields
\[
 \left\|\sup_{\lambda\in J}
 |(M^{a/q}_\lambda-B^{a/q}_{\lambda,N})f|\right\|_2
 \le C_dq^{-d/2}\Sigma_N^{1/2}\|f\|_2.
\]
Summing over $q\le N$ and primitive $a$, and using
$\sum_{q\ge1}q^{1-d/2}<\infty$ for $d\ge5$, proves the required
bound for this last term. Combining the four terms as in
\eqref{tri} and \eqref{cutoffdiff}, and then increasing $J$ to
$\mathcal T_N$, proves \eqref{eq:block-tail}. All constants are
independent of $N,Q,J$, and $\mathbf F$; no lacunarity within a block
has been used.

For the initial count, discarded indices satisfy
$j^2<3\log_2N$. Therefore
\[ K(N)\le\sum_{1\le j<\sqrt{3\log_2N}}2^j
\le2^{1+\sqrt{3\log_2N}}.\]
Writing $x=\log_2N$ and using
$\sqrt{3x}\le\eps x+3/(4\eps)$ proves
\eqref{eq:initial-block-count}, with $C_\eps=2^{1+3/(4\eps)}$.
\end{proof}
\subsubsection{Transference from the continuous case}
Throughout this section, write $\Lambda=\Lambda(\mathbf F)$.
For a set $\mathcal E\subset(0,\infty)$, let
$
 \mathcal M_{\mathcal E}g
 =\sup_{r\in\mathcal E}|\sigma_{r^2}*g|$,
where $\sigma_r$ is the normalized surface measure on the sphere of
radius $r^{1/2}$. The continuous theorem that we need is \cite[Proposition 1.4]{MR1955209}:
if $d\ge3$, $1<p\le p_0$, and
\begin{equation}\label{rw:continuous-entropy}
 \sup_{k\in\Z}\sup_{0<\delta<1}
 \delta^{(d-1)(p-1)}
 N\bigl(\mathcal E\cap[2^k,2^{k+1}),2^k\delta\bigr)<\infty,
\end{equation}
then, without any additional regularity assumption,
\begin{equation}\label{rw:continuous-bound}
 \|\mathcal M_{\mathcal E}g\|_{L^{p,\infty}(\R^d)}
 \le C\|g\|_{L^{p,1}(\R^d)}.
\end{equation}
The constant depends only on $d,p$ and the covering bound in
\eqref{rw:continuous-entropy}. We use only this restricted weak
conclusion. We will apply it to $\mathcal E=\sqrt\Lambda$ to derive
the maximal estimate \eqref{rw:B-bound}.
\begin{lemma}\label{rw:model-bound}
Let $d\ge5$, $1<p\le p_0$, and $\beta=(d-1)(p-1)$.
Suppose that $\Ent_\beta(\mathbf F)<\infty$. Then, with a constant
$C$ depending only on $d,p$, and $\Ent_\beta(\mathbf F)$, there holds

\begin{equation}\label{rw:B-bound}
 \left\|\sup_{\lambda\in\Lambda}|B_{\lambda,N}f|\right\|_{\ell^{p,\infty}}
 \le C\|f\|_{\ell^{p,1}},\quad\forall N\in\N,\ N\ge2.
\end{equation}

\end{lemma}

\begin{proof}
Since $R_{j+1}/R_j>4$, a dyadic interval $[2^k,2^{k+1})$ meets at
most one of the sets $R_jF_j$. If it meets this set, then
$1/2<R_j/2^k<2$. Dilation of a covering of $F_j$ therefore gives
\begin{align}\label{dilation}
N\bigl(\mathcal{E}\cap[2^k,2^{k+1}),2^k\delta\bigr)
 \leq N(R_jF_j,2^k\delta)\leq C\bigl(1+\Ent_\beta(\mathbf F)\bigr)\delta^{-\beta}.
\end{align}
In fact, for $0<\delta<1$, let $\varepsilon=2^k\delta/R_j$. Dilation of
the covering intervals gives $
N(R_jF_j,2^k\delta)=N(F_j,\varepsilon)$.
If $\varepsilon<1$, the definition of $\Ent_\beta(\mathbf F)$ gives
$N(F_j,\varepsilon)
 \leq \Ent_\beta(\mathbf F)\varepsilon^{-\beta}
 \le 2^\beta \Ent_\beta(\mathbf F)\delta^{-\beta}$.
If $\varepsilon\ge1$, a single interval covers $F_j\subset[1,2)$.
 Hence,
\[
 \sup_{k\in\Z}\sup_{0<\delta<1}
 \delta^\beta
 N\bigl(\mathcal{E}\cap[2^k,2^{k+1}),2^k\delta\bigr)
 \le 1+2^\beta \Ent_\beta(\mathbf F).
\]
By \eqref{rw:continuous-bound}, there is a finite constant $W_{\mathcal E}$,
depending only on $d,p$, and $\Ent_\beta(\mathbf F)$, such that
\[
 \|\mathcal M_\mathcal{E} g\|_{L^{p,\infty}}
 \le W_\mathcal{E}\|g\|_{L^{p,1}}.
\]

We use the factorization $  B_{\lambda,N}=T_{\lambda,Q}\circ U_Q$,
where $T_{\lambda,Q}$ and $U_Q$ are convolution operators associated with multipliers (\ref{eq:t-R}) and (\ref{eq:u-N}), respectively.

Let $J\subset\Lambda$ be finite and, for $\lambda\in J$, let $ m_{\lambda,Q}(\xi)
 =\widetilde\psi_{2Q}(\xi)\,
  \widetilde{d\sigma_\lambda}(\xi)$. These are smooth multipliers supported in
$[-1/(4Q),1/(4Q)]^d$ and satisfy 
\[
 \mathcal F^{-1}(m_{\lambda,Q}\widetilde g)(x)
 =\sigma_{\lambda}\bigstar(\psi_{2Q}\bigstar g)(x).
\]
Therefore, we obtain
\begin{align*}
 \left\|\sup_{\lambda\in J}
       |\mathcal F^{-1}(m_{\lambda,Q}\widetilde g)|\right\|_{L^{p,\infty}}
 &\le\|\mathcal M_\mathcal{E}(\psi_{2Q}\bigstar g)\|_{L^{p,\infty}}\\
 &\le W_\mathcal{E}\|\psi_{2Q}\bigstar g\|_{L^{p,1}}
 \le W_\mathcal{E}\|\psi_{2Q}\|_{L^1}\|g\|_{L^{p,1}},
\end{align*}
in which the constant is independent of $Q$ and $J$. The periodization
$\sum_{\ell\in\Z^d}m_{\lambda,Q}(\xi-\ell/Q)$ is the
multiplier of $T_{\lambda,Q}$, and hence Theorem~\ref{thm:sampling} gives
\[
 \left\|\sup_{\lambda\in J}|T_{\lambda,Q}f|\right\|_{\ell^{p,\infty}}
 \le C_{d,p}W_{\mathcal E}\|\psi_{2Q}\|_{L^1}\|f\|_{\ell^{p,1}}.
\]

Let $K^U_Q$ be the convolution kernel of $U_Q$. 
By Minkowski's inequality and \eqref{L^1normandcounting},
\begin{align*}
 \|U_Qf\|_{\ell^{p,1}}
 \le\sum_{z\in\Z^d}|K^U_Q(z)|
                    \|f(\,\cdot-z)\|_{\ell^{p,1}}=\|K^U_Q\|_{\ell^1}\|f\|_{\ell^{p,1}}
 \le C_d\|f\|_{\ell^{p,1}},
\end{align*}
which, together with $B_{\lambda,N}=T_{\lambda,Q}\circ U_Q$, gives
\[
 \left\|\sup_{\lambda\in J}|B_{\lambda,N}f|\right\|_{\ell^{p,\infty}}
 \le C_{d,p}W_{\mathcal E}\|f\|_{\ell^{p,1}},
\]
since $\|\psi_{2Q}\|_1=\|\psi\|_1$. The constant is independent of
$N,Q,J$, and $\#J$. Increasing finite sets $J$ to $\Lambda$ and using
continuity from below of the level sets proves \eqref{rw:B-bound}.
\end{proof}

\subsubsection{Completion of the proof}
Assume that $\Ent_\beta(\mathbf F)<\infty$ and let $H\subset\Z^d$
be finite. Since every $A_\lambda$ is a probability average,
$0\le M_\Lambda\one_H\le1$, so it suffices to consider
$0<\alpha<1$.

For an integer $N\ge2$, split the radii into
$\Lambda\setminus\mathcal T_N$ and $\mathcal T_N$. The initial
supremum is dominated by the sum of its $K(N)$ averages. Mass
preservation and Chebyshev's inequality give
\[
 \#\{x:\sup_{\lambda\in\Lambda\setminus\mathcal T_N}
        |A_\lambda\one_H(x)|>\alpha/2\}
 \le 2\alpha^{-1}K(N)\#H.
\]
On the tail use the identity
\[
 A_\lambda=\rho_\lambda B_{\lambda,N}
          +\rho_\lambda(\overline A_\lambda-B_{\lambda,N}).
\]
As $\rho_\lambda\le R_d^{\rm norm}$, the remaining level set is
contained in the union of
\[
 \left\{\sup_{\lambda\in\mathcal T_N}|B_{\lambda,N}\one_H|
          >\frac\alpha{4R_d^{\rm norm}}\right\}\quad\text{and}\quad 
 \left\{\sup_{\lambda\in\mathcal T_N}
       |(\overline A_\lambda-B_{\lambda,N})\one_H|
          >\frac\alpha{4R_d^{\rm norm}}\right\}.
\]
Lemmas~\ref{rw:model-bound} and \ref{lem:block-tail} therefore yield
\begin{equation}\label{rw:sufficient-split}
 \#\{x:M_\Lambda\one_H(x)>\alpha\}
 \le C\#H\left[\alpha^{-1}K(N)+\alpha^{-p}
                         +\alpha^{-2}N^{-(d-4)}\right],
\end{equation}
where $C$ depends only on $d,p$, and $\Ent_\beta(\mathbf F)$.

We now choose $N$ and $\eps$ so that the contributions of the initial
radii and the approximation error in \eqref{rw:sufficient-split}
are bounded by $C\alpha^{-p}\#H$. Since $p\le p_0<2$, set
\[
 a=\frac{2-p}{d-4}>0,\qquad
 N=\left\lceil2\alpha^{-a}\right\rceil\quad\text{and}\quad 
 \eps=\frac{p-1}{2a}>0.
\]
Then $2\alpha^{-a}\le N\le3\alpha^{-a}$, and hence
\[
 \alpha^{-2}N^{-(d-4)}
 \le2^{-(d-4)}\alpha^{-2+a(d-4)}
 =2^{-(d-4)}\alpha^{-p}.
\]
By \eqref{eq:initial-block-count}, we have
\[
 \alpha^{-1}K(N)
 \le C_\eps3^\eps\alpha^{-1-a\eps}
 =C_\eps3^\eps\alpha^{-(p+1)/2}
 \le C_\eps3^\eps\alpha^{-p}.
\]
For this choice of $\eps$, the constant $C_\eps3^\eps$ depends on
$d,p$; no uniformity as $p\downarrow1$ is asserted. Substitution in
\eqref{rw:sufficient-split} proves \eqref{eq:restricted-type}.

\subsection{Proof of the necessity part of Theorem~\ref{thm:restricted}}

We derive the covering condition by testing the restricted weak type
estimate on the indicator of a lattice ball. A separated set of radii
gives disjoint annuli, allowing us to sum the masses of the corresponding
averages.

\begin{proof}
Suppose that \eqref{eq:restricted-type} holds with constant $C_0$.
Choose $j_d$, depending only on $d$, so that
$n_j!/20\ge\max\{1,\sqrt d\}$ for $j\ge j_d$. Fix such a $j$ and let
$2^{-j-1}\le\delta<1$. Choose a maximal $\delta/2$-separated subset
$P\subset F_j$. By maximality, $N(F_j,\delta)\le\#P$.
We may assume $P\ne\varnothing$, since an empty block contributes zero.

Set
\[
 h={\delta R_j}/{10},\quad
 H=\{x\in\Z^d:|x|\le h\}\quad\text{and}\quad S_t=\{x\in\Z^d:\big||x|-R_jt\big|\le h\}.
\]
Since $h\ge n_j!/20$, our choice of $j$ ensures that
$h\ge\max\{1,\sqrt d\}$. For each $t\in P$, the support of
$A_{(R_jt)^2}\one_H$ is contained in $S_t$. Since the corresponding radii are
separated by at least $\delta R_j/2=5h$, the annuli $S_t$ are pairwise disjoint.
Let $S:=\bigsqcup_{t\in P}S_t$. Then
\begin{equation}\label{rw:volumes}
 \#H\simeq_d h^d
 \quad\text{and}\quad
 \#S\lesssim_d (\#P) hR_j^{d-1}.
\end{equation}
Here we compare lattice points with unit cubes and use
$h\ge\max\{1,\sqrt d\}$. Since the annuli are disjoint and the averages preserve mass, we have
\begin{equation}\label{rw:mass}
 \#P\cdot\#H
 =\sum_{t\in P}\sum_{x\in\Z^d}A_{(R_jt)^2}\one_H(x)=\sum_{t\in P}\sum_{x\in S_t}A_{(R_jt)^2}\one_H(x)
 \le\sum_{x\in S}M_\Lambda\one_H(x).
\end{equation}
On the other hand, the restricted weak estimate yields
\begin{align*}
 \sum_{x\in S}M_\Lambda\one_H(x)
 \leq\int_0^\infty
       \min\{\#S,C_0\#H\,\alpha^{-p}\}\,d\alpha=p' C_0^{1/p}(\#H)^{1/p}(\#S)^{1/p'}.
\end{align*}
Combining this with \eqref{rw:mass}, raising to the power $p$,
and applying \eqref{rw:volumes}, we obtain
$$(\#P)^p\cdot(\#H)^p
 \leq C_{d,p}C_0\,\#H\,(\#P)^{p-1}
          (hR_j^{d-1})^{p-1}.$$
Since $\#H\simeq_dh^d$, it follows that
$\#P\cdot\left( h/{R_j}\right)^{(d-1)(p-1)}
 \le C_{d,p}C_0$.
Together with $N(F_j,\delta)\le\#P$ and $h=\delta R_j/10$,
this gives, for all $ 2^{-j-1}\le\delta<1$,
\begin{align*}
    \delta^\beta N(F_j,\delta)\le C_{d,p}C_0.
\end{align*}

If $0<\delta<2^{-j-1}$, the grid condition on $F_j$ gives $N(F_j,\delta)=\#F_j=N(F_j,2^{-j-1})$, and hence
\[
\delta^\beta N(F_j,\delta)
 \le2^{-(j+1)\beta}N(F_j,2^{-j-1})\le C_{d,p}C_0,
\]
by the estimate already proved at $\delta=2^{-j-1}$. The finitely
many smaller indices $j<j_d$ omitted above satisfy
$\delta^\beta N(F_j,\delta)\le\#F_j\le2^{j_d}$ for all
$0<\delta<1$, where $j_d$ depends only on $d$. Consequently
$\Ent_\beta(\mathbf F)\le C_{d,p}(1+C_0)<\infty$.
\end{proof}

For completeness, the Lorentz formulation in Theorem~\ref{thm:restricted}
follows from positivity. For $1<p<\infty$, the functional
\[
 \|u\|_{p,\infty}^{\#}
 :=\sup_{0<\#S<\infty}(\#S)^{-1/p'}\sum_{x\in S}|u(x)|
\]
satisfies
$\|u\|_{p,\infty}\le\|u\|_{p,\infty}^{\#}
 \le p'\|u\|_{p,\infty}$.
The layer-cake inequality
$M_\Lambda f\le\int_0^\infty M_\Lambda\one_{\{|f|>s\}}\,ds$
and \eqref{eq:restricted-type} thus give
\[
 \|M_\Lambda f\|_{p,\infty}
 \le p'C_0^{1/p}\int_0^\infty\#\{|f|>s\}^{1/p}\,ds
 =p'C_0^{1/p}\|f\|_{p,1}.
\]
First apply this to finitely supported functions and finite radius families;
truncation and increasing finite suprema give the full assertion.
The converse follows by taking $f=\one_H$.

\subsection{Critical exponents}

We combine Theorem~\ref{thm:restricted} with interpolation to obtain
the characterization in Corollary~\ref{cor:critical}.

\begin{proof}[Proof of Corollary~\ref{cor:critical}]
We first record the interpolation used here. If $M_\Lambda$ maps
$\ell^{r,1}$ to $\ell^{r,\infty}$ with norm $W_r$, $r>1$, then its
contraction property on $\ell^\infty$ implies strong type $(p,p)$
for every $r<p<\infty$. Indeed, real interpolation with parameter
$\eta=1-r/p$ and interpolation index $p$ uses
\[
 (\ell^{r,1},\ell^\infty)_{\eta,p}
  =(\ell^{r,\infty},\ell^\infty)_{\eta,p}
  =\ell^p
\]
with equivalence constants depending only on $r,p$, and gives
\[
 \|M_\Lambda f\|_p\le C_{r,p}W_r^{r/p}\|f\|_p.
\]
For a finite radius family, apply this interpolation to each linearization
$f\mapsto A_{\lambda(x)}f(x)$, uniformly in the selector $\lambda(x)$.
For each quantity, a pointwise maximizing selector recovers the finite
maximal function. Letting the finite radius families increase to $\Lambda$ gives
the assertion.
No uniformity as $p\downarrow r$ is claimed.

Suppose \eqref{eq:critical-conditions} holds, with
$\beta=(d-1)(p_*-1)$. Theorem~\ref{thm:restricted} gives restricted
weak type at $p_*$ and its failure at every $1<p<p_*$, hence also
failure of strong type in that range. The preceding interpolation
with $r=p_*$ gives strong type for $p>p_*$; $p=\infty$ follows from
the probability normalization.

The infinite entropy at smaller exponents forces $\Lambda$ to be
infinite, since only finitely many nonempty finite blocks would give
finite entropy at every positive exponent. Distinct squared radii
have disjoint lattice spheres, and therefore
\[
 \|M_\Lambda\one_{\{0\}}\|_1
 =\sum_{\lambda\in\Lambda}\sum_{|x|^2=\lambda}\frac1{s_\lambda}
 =\sum_{\lambda\in\Lambda}1=\infty.
\]
This proves the separate failure of strong type $(1,1)$.

Conversely, restricted weak type at $p_*$ implies
$\Ent_\beta(\mathbf F)<\infty$ by Theorem~\ref{thm:restricted}.
If $\Ent_\theta(\mathbf F)<\infty$ for some $0<\theta<\beta$,
the same theorem gives restricted weak type at
$r=1+\theta/(d-1)<p_*$. Interpolation then gives strong type for any
$r<p<p_*$, contrary to the assumed unboundedness. This proves both
directions of the corollary.
\end{proof}

\subsection{A binary examples}
We now give a family for each critical exponent in
$(1,p_0]$. Fix $0<\beta\le1$ and set
\[
 I_\beta=\{i\ge1:\lfloor\beta i\rfloor
                       -\lfloor\beta(i-1)\rfloor=1\}.
\]
Define the set
\begin{equation}\label{rw:binary-block}
 F_j=\Bigg\{1+\sum_{\substack{1\le i\le j\\i\in I_\beta}}
          \varepsilon_i2^{-i}:\varepsilon_i\in\{0,1\}\Bigg\}.
\end{equation}
There are $\lfloor\beta j\rfloor$ free binary digits, so
$F_j\subset[1,2)\cap2^{-j}\Z$ and
$\#F_j=2^{\lfloor\beta j\rfloor}$. Fix $0<\delta<1$ and put
$k=\lceil\log_2(1/\delta)\rceil$, so
$2^{-k}\le\delta<2^{-k+1}$. If $k\le j$, grouping the points of $F_j$ by their first $k$
binary digits gives a cover by $2^{\lfloor\beta k\rfloor}$ intervals
of length $2^{-k}$. If $k>j$, covering each point separately
gives a cover by $2^{\lfloor\beta j\rfloor}$ intervals of length $\delta$.
In either case,
\begin{equation}\label{rw:binary-entropy}
 N(F_j,\delta)\le2^{\beta k}\le2^\beta\delta^{-\beta},
 \quad\text{and hence}\quad
 \Ent_\beta(\mathbf F)\le2^\beta<\infty.
\end{equation}
At $\delta_j=2^{-j-1}$, each covering interval contains at most
one point of the grid. For every $0<\theta<\beta$ it follows that
\[
 \delta_j^\theta N(F_j,\delta_j)
 =2^{-\theta(j+1)}2^{\lfloor\beta j\rfloor}
 \ge2^{-1-\theta}2^{j(\beta-\theta)}\longrightarrow\infty.
\]
Thus $\Ent_\theta(\mathbf F)=\infty$, and
Corollary~\ref{cor:critical} applies with
$p_*=1+\beta/(d-1)$. When $\beta=1$, all digits are free,
$F_j=[1,2)\cap2^{-j}\Z$, and $p_*=p_0$.
The conclusion at $p_*$ is restricted weak type; no general weak type
or strong type assertion at that exponent is made here.

\medskip

\noindent{\bf AI Disclosure:}
Generative AI tools were used to assist with language editing,
mathematical verification, and revisions of the arguments and proofs.
\medskip

\noindent {\bf Acknowledgements:} The authors thank Kevin Hughes for bringing related work to our attention.
S. Lee is supported by the National Research Foundation of Korea (NRF) through Grant No. RS-2024-00342160.  J. Li is supported by ARC DP260100485.  C.W. Liang and C.Y. Shen are supported by NSTC through grant 111-2115-M-002-010-MY5. Part of this work was carried out while
the third author was a student at Macquarie University. He thanks the university for its support and a comfortable working environment. He also thanks HaoYun Yao for a helpful conversation.

\end{document}